\documentclass[11pt]{article}

\usepackage{amsfonts,amsmath,amsthm,color,enumerate,parskip}
\usepackage{hyperref}
\usepackage[margin=1.9 cm,bottom=16mm,footskip=10mm,top=19mm]{geometry}

\newtheorem{thm}{Theorem}[section]
\newtheorem{dfn}[thm]{Definition}
\newtheorem*{dfn*}{Definition}
\newtheorem{lem}[thm]{Lemma}
\newtheorem{cor}[thm]{Corollary}
\newtheorem{prop}[thm]{Proposition}

\DeclareMathOperator{\Bin}{Bin}
\newcommand{\simple}{$\mathcal{T}_{cliq}$-forcing}
\newcommand{\esimple}{$\mathcal{T}_{bip}$-forcing}

\newcommand{\floor}[1]{
    \left \lfloor #1 \right \rfloor
}

\title{Tournament quasirandomness from local counting}

\author{Matija Buci\'c\thanks{Department of Mathematics, ETH, 8092 Z\"urich, Switzerland. Email: \texttt{matija.bucic@math.ethz.ch}}
\and Eoin Long\thanks{School of Mathematics, University of Birmingham, Birmingham, UK. Email: \texttt{e.long@bham.ac.uk}}
\and Asaf Shapira\thanks{School of Mathematics, Tel Aviv University, Tel Aviv 69978, Israel. Email: \texttt{asafico@tau.ac.il}. \newline \hspace*{1.5em} Supported in part by ISF Grant 1028/16 and ERC Starting Grant 633509.}
\and Benny Sudakov\thanks{Department of Mathematics, ETH, 8092 Z\"urich, Switzerland. Email: \texttt{benjamin.sudakov@math.ethz.ch}.\newline \hspace*{1.5em} Research supported in part by SNSF grant 200021-175573.}}
\date{}

\begin{document}

\maketitle

\abstract{
A well-known theorem of Chung and Graham states that if $h\geq 4$ then a tournament $T$ is quasirandom if and only if $T$ contains each $h$-vertex tournament the `correct number' of times as a subtournament. In this paper we investigate the relationship between quasirandomness of $T$ and the count of a \emph{single} $h$-vertex tournament $H$ in $T$. We consider two types of counts, the global one and the local one.\vspace{1mm}

We first observe that if $T$ has the correct {\em global} count of $H$ and $h \geq 7$ then quasirandomness of $T$ is only forced if $H$ is transitive. The next natural question when studying quasirandom objects asks whether possessing the correct {\em local} counts of $H$ is enough to force quasirandomness of $T$. A tournament $H$ is said to be locally forcing if it has this property.\vspace{1mm}

Variants of the local forcing problem have been studied before in both the graph and hypergraph settings. Perhaps the closest analogue of our problem was considered by Simonovits and S\'os who looked at whether having `correct counts' of a fixed graph $H$ as an induced subgraph of $G$ implies $G$ must be quasirandom, in an appropriate sense. They proved that this is indeed the case when $H$ is regular and conjectured that it holds for all $H$ (except the path on 3 vertices).
Contrary to the Simonovits-S\'os conjecture, in the tournament setting we prove that a constant proportion of all tournaments are not locally forcing. In fact, any locally forcing tournament must itself be strongly quasirandom. On the other hand, unlike the global forcing case, we construct infinite families of non-transitive locally forcing tournaments.}

\section{Introduction}

A combinatorial structure is said to be `quasirandom' if it behaves in a similar manner to a random structure, where the comparison is made with respect to some deterministic property. The systematic study of quasirandomness was initiated by Thomason \cite{Th1} \cite{Th2}, and Chung, Graham and Wilson \cite{CGW}, who examined notions of quasirandomness arising from various graph properties. One of the surprising conclusions of these papers is that a wide-range of natural graph properties all lead to essentially the same notion of quasirandomness, in the sense that a graph satisfying one of the properties necessarily satisfies them all. Since then, notions of quasirandomness have been extensively studied in a wide variety of contexts, including hypergraphs \cite{CGsetsystems}, \cite{GowHR}, \cite{NRS}, \cite{RS}, permutations \cite{Cooper}, \cite{KP} and groups \cite{Gow}. The reader is referred to the survey \cite{KS} for an overview of this extensive topic. \vspace{1mm}

In this paper we will study notions of quasirandomness for tournaments. The first paper on this topic is due to Chung and Graham \cite{CG} who proved that, as in the graph case, a wide range of natural tournament properties give rise to the same notion of quasirandomness. Before stating some of their results we require a little notation. 
A tournament $T = (V,E)$ consists of a set of vertices $V  = V(T)$, together with a set of edges $E  = E(T) \subset V \times V$, with the property that (i) $(u,u) \notin E$ for all $u\in V$ and (ii) exactly one of $(u,v)$ or $(v,u)$ lies in $E$ for all distinct $u,v \in V$. We often write $\overrightarrow {uv}$ to denote an edge $(u,v)$. A tournament $H$ appears as a subtournament of $T$ if there is a map $\phi: V(H) \to V(T)$ such that $\overrightarrow {uv} \in E(H)$ if and only if $\overrightarrow {\phi (u)\phi (v)} \in E(T)$. The map $\phi $ is said to be a labelled embedding of $H$ into $T$. Let 
	\begin{equation*}
		N^*_T(H)  
			= 
		\big  | \big  \{ \phi : V(H) \to V(T): 
		\phi \mbox{ is a labelled embedding of } H \mbox { into } T \big  \} \big |.
	\end{equation*} 
Given $U \subset V(T)$ let $T[U]$ denote the subtournament of $T$ induced by the vertex set $U$. Also let $N^*_T(H;U) := N^*_{T[U]}(H)$. For $u \in V(T)$ let $d^+_T(u) = |\{v\in V(T): \overrightarrow{uv} \in E(T)\}|$ and $d^-_T(u) = |\{v\in V(T): \overrightarrow{vu} \in E(T)\}|$. A tournament $T$ is regular if $d^+_T(u) = d^-_T(u)$ for all $u\in V(T)$. For $U,W \subseteq V(T)$ we denote by $e(U,W)$ the number of edges starting in $U$ and ending in $W$.

We say that $T$ is an $n$-vertex tournament if $|V(T)| = n$. An ordering of $T$ is a bijective map $\sigma : V \to [n]$ and the set of all orderings of $T$ is naturally identified with $S_n$, the symmetric group on $n$ elements. An edge $\overrightarrow {uv} \in E$ is a $\sigma $-forward edge if $\sigma (u) < \sigma (v)$ and we write $F_{\sigma ,T} \subset E$ to denote the set of $\sigma $-forward edges of $T$. Let $B_{\sigma , T} = E \setminus F_{\sigma ,T}$, the  set of $\sigma $-backward edges. $T$ is said to be transitive if $F_{\sigma ,T} = E$ for some $\sigma \in S_n$ and write $Tr_n$ to denote the unique (up to isomorphism) $n$-vertex transitive tournament.  Lastly, we will write $a \pm b$ to denote some value $c$ with $a-b \leq c \leq a+b$.\vspace{1mm}

Our starting point in this paper is a result due to Chung and Graham \cite{CG}, which gives two equivalences of tournament quasirandomness. 

\begin{thm}[Chung--Graham]
	\label{thm: chung-graham}
	Let $h \in {\mathbb N}$ with $h\geq 4$. 
	Then for any $n$-vertex tournament $T$, the following 
	properties are equivalent:
\begin{itemize}
	\item ${\cal P}_1$: $|F_{\sigma ,T}| = 
	\frac {1}{2} \binom {n}{2} \pm  o(n^2)$ for every ordering 
	$\sigma $ of $T$.
	\item ${\cal P}_2(h)$: $N^*_T(H) = 2^{-\binom {h}{2}}n^h \pm 
	o(n^h)$ for every $h$-vertex tournament $H$.
\end{itemize}
\end{thm}

In fact, there are nine further equivalences given in this paper (see also \cite{Griffiths}, \cite{KSh}). Equivalence here is understood in the following sense, focusing on the implication ${\cal P}_1 \implies {\cal P}_2(h)$: given $\varepsilon >0$ there is $\delta >0$ and $n_0 \in {\mathbb N}$ so that if $T$ is an $n$-vertex tournament which satisfies $|F_{\sigma ,T}| = \frac {1}{2} \binom {n}{2} \pm \delta n^2$ for every ordering $\sigma $ of $T$ and $n\geq n_0$ then $N^*_T(H) = 2^{-\binom {h}{2}}n^h \pm \varepsilon n^h$ for every  $h$-vertex tournament $H$. \vspace{1mm} 

It is easily seen that both properties ${\cal P}_1$ and ${\cal P}_2(h)$ hold for a random $n$-vertex tournament $T$ with high probability. We will say $T$ is \emph{quasirandom} if it satisfies ${\cal P}_1$, that is $|F_{\sigma ,T}| = \frac {1}{2}\binom {n}{2} \pm o(n^2)$. In light of Theorem \ref{thm: chung-graham}, this notion is equivalent to the analogous notion arising from ${\cal P}_2(h)$. \vspace{1mm}

\subsection{Globally forcing tournaments}
Although ${\cal P}_2(h)$ guarantees that $T$ is quasirandom, it is natural to ask whether this can already be deduced from the count of a \emph{single} tournament $H$. That is, does quasirandomness of $T$ already follow if $N^*_T(H) = 2^{-\binom {h}{2}}n^h \pm o(n^h)$ for a {single} $h$-vertex tournament $H$? 

\begin{dfn*} 
	Let $H$ be an $h$-vertex tournament. Given an $n$-vertex tournament $T$, consider 
	the following property: 
	\begin{itemize}
		\item ${\cal P}_2(H)$: $N^*_T(H) = 2^{-\binom {h}{2}}n^h \pm o(n^h)$.
	\end{itemize} 
	The tournament $H$ is said to be \emph{globally forcing} if ${\cal P}_2(H) \implies 
	{\cal P}_1$. 
\end{dfn*}

It follows easily from exercise 10.44(b) of \cite{lovasz-book} that for $h\geq 4$ each transitive tournament $Tr_h$ is globally forcing. This statement was recently reproved in the language of flag-algebras by Coregliano and Razborov \cite{CR}. Our first observation is that for $h \geq 7$, the transitive $h$-vertex tournament is the \textit{only} tournament with this property.

\begin{prop}
	\label{prop: forcing iff transitive}
	Let $H$ be an $h$-vertex tournament 	with $h\geq 7$. Then $H$ is globally forcing 
	if and only if $H$ is transitive.
\end{prop}

We remark that, while the condition that $h\geq 7$ may seem unusual in Proposition \ref{prop: forcing iff transitive}, it has been proven by Coregliano, Parente and Sato \cite{CPS} that there is a non-transitive globally forcing tournament $H$ on five vertices (see Theorem 1.1 from \cite{CPS}). It would perhaps be interesting to determine which, if any, of the remaining small tournaments are also globally forcing.

\subsection{Locally forcing tournaments}
We have seen that having `correct count' of a fixed tournament $H$ is not enough to guarantee quasirandomness, for essentially any non-transitive $H$. This is a fairly common situation when studying quasirandom properties in general and the key insight to understand why came from Simonovits and S\'os \cite{SS1} who observed that quasirandomness is a hereditary property, in the sense that any large subgraph of a random structure must also be random-like. 

This leads us to the natural question of whether requiring that $T$ contains the `correct count' of $H$ in all large subsets of $V(T)$ is sufficient to guarantee quasirandomness of $T$. To be more precise, consider the following definition.

\begin{dfn*} 
	Let $H$ be an $h$-vertex tournament. Given an $n$-vertex tournament $T$, consider 
	the following property: 
	\begin{itemize}
		\item ${\cal P}^*_2(H)$: $N^*_T(H;U) = 2^{-\binom {h}{2}}|U|^h \pm o(n^h)$ for 
		every set $U \subset V(T)$.
	\end{itemize} 
	The tournament $H$ is said to be
	\emph{locally forcing} if ${\cal P}^*_2(H) \implies {\cal P}_1$.
\end{dfn*}

An analogous property for graphs was studied by Simonovits and S\'os for both induced \cite{SS1} and not-necessarily induced subgraphs \cite{SS2} as well as for hypergraphs by Conlon, H\`an, Person and Schacht \cite{CHPS} and Dellamonica and R\"odl \cite{DR}. In case of graphs the not-necessarily induced case turned out to be much easier and was resolved by Simonovits and S\'os \cite{SS2} who showed that all non-empty graphs must be `locally forcing'. This was recently reproved, without use of the regularity lemma, by Conlon, Fox and Sudakov \cite{CFS}. On the other hand the induced case is still very much open and seems to be the closer analogue to our problem. Indeed, in the case of embedding a tournament one needs to ensure that each pair of vertices get mapped to an edge with one of two possible orientations, while in the induced case each pair needs to be sent to one of two possible states (is an edge or is not an edge). One of the main conjectures in the area, due to Simonovits and S\'os \cite{SS2} says that all graphs on $4$ or more vertices must be `locally forcing' in the appropriate induced sense. They prove their conjecture holds for regular graphs and for various other families of graphs.

In the rest of this subsection we present our main results, which give a good understanding of locally forcing tournaments and show a surprisingly different behaviour compared to the one conjectured by Simonovits and S\'os in the graph case. Our first result shows that in order for a tournament to be locally forcing it must be quite strongly quasirandom in several ways.

\begin{thm} \label{thm: locally-forcing} 
		Any non-transitive, locally forcing $h$-vertex tournament $H$ satisfies 
		\begin{enumerate}[(i)]
			\item 
				\begin{equation*}
					\sum _{v \in V(H)} \big (d^+_H(v) - d^-_H(v)\big )^2 \leq h(h-1).
				\end{equation*}
			\item $|F_{\sigma, H}|=\frac{1}{2}\binom{h}{2} \pm h^{3/2}\sqrt{\log h}$ for every ordering $\sigma$ of $H$. 
			\item For any disjoint subsets $U,W \subseteq T$ we must have $|e(U,W)|\le \frac{1}{2}|U||W| + 2h^{3/2}$ 
		\end{enumerate}
\end{thm}

Given that Theorem \ref{thm: locally-forcing} says that any locally forcing graph must be strongly quasirandom a natural guess for an example might be to take the random tournament, obtained by orienting each edge of a complete graph uniformly at random, independently between edges. Our next result shows that quite surprisingly with positive probability the random tournament fails condition (i) of Theorem \ref{thm: locally-forcing}.

\begin{cor}\label{cor:random-is-not-locally-forcing}
The random tournament is not locally forcing with positive probability.
\end{cor}

This tells us that a positive proportion of all tournaments are in fact not locally forcing, in stark contrast to the induced graph problem in which Simonovits-S\'os conjecture states that all graphs on $4$ or more vertices should be locally forcing. In particular, one can easily find a regular tournament which fails conditions (ii) or (iii) whereas in the induced graph case any regular graph is known to be locally forcing.

This might suggest that there are essentially no non-transitive locally forcing tournaments, as in the global forcing case. We show that this is also not the case and find many non-transitive examples of locally forcing tournaments.

\begin{thm}\label{thm: exist-locally-forcing}
    For any large enough $h$ there exist non-transitive, $h$-vertex, locally forcing tournaments.
\end{thm}

Our examples of locally forcing tournaments come from the following random construction. We take a Steiner triple system on $h$ vertices, that is a partition of the edge set of the complete graph on $h$ vertices into edge disjoint triangles. Now we orient each of these triangles into one of two cycles of length three uniformly at random and independently between triangles. We show that this produces with high probability a locally forcing tournament. It is well known (see \cite{Kirkman}) that Steiner triple systems exist provided $h\equiv 1,3 \pmod{6}$ and our arguments allow us a lot of freedom to modify the above procedure in order to obtain examples for other $h$ as well. 

All our results are based on a characterisation of locally forcing tournaments in terms of certain graph polynomials. In order to motivate how these polynomials arise let us describe two natural candidates for tournaments with correct local counts of a tournament $H$, which are not quasirandom.

\begin{dfn*}Given $\alpha \in [0,1]$, let ${\cal T}(n,\alpha )$ denote the distribution on the set of tournaments with vertex set $\{1,\ldots, n\}$ in which each pair $\{i,j\}$ with $1 \leq i < j \leq n$ appears independently as $\overrightarrow {ij}$ with probability $\alpha$. We denote by $\mathcal{T}_{cliq}$ the collection of all distributions ${\cal T}(n,\alpha )$ as $\alpha$ varies.
\end{dfn*}
\noindent Loosely speaking we say that a graph $H$ is \simple{} if no ${\cal T}(n,\alpha )$ shows that $H$ is not locally forcing. In other words if there is no $\alpha \neq 1/2$ such that ${\cal T}(n,\alpha )$, which is clearly not quasirandom, has the same local counts of $H$ as ${\cal T}(n,1/2)$ (which {\em is} quasirandom).

\begin{dfn*} Given $\alpha \in [0,1]$, let ${\cal T}(n,n, \alpha)$ denote the distribution on the set of tournaments with vertex set $\{1,\ldots, 2n\}$ in which both $\{1,\ldots, n\}$ and $\{n+1,\ldots, 2n\}$ are oriented according to $\mathcal{T}(n,1/2)$ and each pair $\{i,j\}$ with $1 \leq i \leq n, n+1 \le j \le 2n$ appears as $\overrightarrow {ij}$ with probability $\alpha$, with all choices made independently. We denote by $\mathcal{T}_{bip}$ the collection of all distributions ${\cal T}(n,n,\alpha )$ as $\alpha$ varies.
\end{dfn*}
\noindent Loosely speaking a graph $H$ is \esimple{} if no ${\cal T}(n,n,\alpha )$ shows that $H$ is not locally forcing. I.e.\ if there is no $\alpha \neq 1/2$ such that ${\cal T}(n,n,\alpha )$, which is not quasirandom, has the same local counts of $H$ as ${\cal T}(n,n,1/2)={\cal T}(2n,1/2)$.

Both properties \simple{} and \esimple{} can be expressed in terms of certain polynomial equations depending on the graph $H$. Our actual definitions of both properties go directly via these polynomial equations and we refer the reader to Section \ref{sec: counting polynomial} for exact definitions.

Quite remarkably it turns out that if $H$ is not locally forcing one must be able to find an $\alpha$ such that either ${\cal T}(n,\alpha )$ or ${\cal T}(n,n,\alpha )$ show $H$ is not locally forcing. In particular, we show:

\begin{thm}\label{thm:-simple and esimple iff locally forcing}
$H$ is locally forcing if and only if it is \simple{} and \esimple{}.
\end{thm}

This result allows us to answer the question of whether a fixed graph $H$ is locally forcing or not in terms of whether certain polynomial equations have a common real root in a bounded interval. This can be answered using a combination of Sturm's Theorem (see \cite{BPR}) and the Euclidean Algorithm for computing the greatest common divisor of a sequence of polynomials. See the concluding remarks for more details.\vspace{1mm}

\noindent \textbf{Organisation of the paper:} The short proof of Proposition \ref{prop: forcing iff transitive} is given in the next section. In Section \ref{sec: counting polynomial} a counting polynomial associated with \simple{} is introduced. In the same section we show that a locally forcing tournament must be \simple{} and that this property imposes many restrictions on $H$, which gives us a proof of Theorem \ref{thm: locally-forcing}. Despite this in Section \ref{sec: simple tournaments} we find many \simple{} tournaments. In Sections \ref{sec: reg simple tourns are forcing} and \ref{sec: proof of simple and esimple implies local forcing} we introduce the polynomials related to the property of \esimple{}, show that any nearly regular tournament must be \simple{} and show Theorem \ref{thm:-simple and esimple iff locally forcing}. Using this we prove that many of our examples of \simple{} tournaments are in fact locally forcing.\vspace{2mm}

\noindent \textbf{Notation:} Before closing this section we introduce some further notation. Given $m,n \in {\mathbb N}$, let $[n] = \{1,\ldots ,n\}$ and $(n)_m$ denote the falling factorial $(n)_m = n(n-1)\cdots (n-m+1)$. Given a set $X$ we write $\binom{X}{k}$ for the collection of $k$-element subsets of $X$. Given a probability distribution ${\cal D}$ on $X$ we write $x \sim {\cal D}$ to denote an element $x\in X$ selected randomly according to ${\cal D}$. We simply write $x \sim X$ when ${\cal D}$ is taken to be the uniform distribution on $X$. All our logarithms are natural so in base $e$.

\section{Globally forcing tournaments}

In this section we will give the short proof of Proposition \ref{prop: forcing iff transitive}. We first show that asymptotically the number of transitive subtournaments of order $h$ in a tournament of order $n$ is minimised by the random tournament. This is precisely the content of Exercise 10.44(b) of \cite{lovasz-book}. We need a variation of this proof, which allows us to also show the fact that any asymptotic minimiser must be quasirandom. This implies that $Tr_h$ is globally forcing for $h \ge 4$ showing the `if' part of the statement of Proposition \ref{prop: forcing iff transitive}.

\begin{proof}[Proof of Proposition \ref{prop: forcing iff transitive}]
	We will first prove that $N^*_T(h):= N^*_T(Tr_h)$ satisfies 
	$N^*_T(h) \geq f_h(n)=(2^{-\binom {h}{2}}-o(1) \big ) { n^h }$ for all $n$-vertex 
	tournaments $T$ where 
	$$
	f_h(n)=
        \begin{cases}
        \prod_{j=0}^{h-1}\left(\frac{n+1}{2^j}-1\right) \quad &\text{ if } n\ge 2^{h-1}-1\\
        0 \quad &\text{ else.}
        \end{cases}
    $$

	We will prove this by induction on $h$. For $h=1$ and $2$ the result is 
	immediate. Furthermore, 
		\begin{equation}
			\label{equation: special case k=3}
			N^*_T(3) 
					= 
			\sum _{v \in V(T)} \binom {d^+(v)}{2} 
					\geq
			n \binom {\frac {n-1}{2}}{2} = n(n-1)(n-3)/8=f_3(n).
		\end{equation}	
	Now assume $h\geq 4$ and $n \ge 2^{h-1}-1$ as otherwise the result is trivial. For each edge $e = \overrightarrow{uv} \in E(T)$ let 
	$N(e) = \{w \in V(T): \overrightarrow{uw}, \overrightarrow{vw} \in E(T)\}$.
	Clearly
		\begin{equation}
			\label{equation: directed triangle count}
			N^*_T(3) = \sum _{e \in E(T)} 
			|N(e)|.
		\end{equation}
	Letting $T[N(e)]$ be the subtournament of $T$ with vertex 
	set $N(e)$, by induction on $h$
		\begin{equation*}
			N^*_T(h) = 
			\sum _{e\in E(T)} 
			N^*_{T[N(e)]}({h-2})
				\geq 
			\sum _{e \in E(T)} 
			 f_{h-2}(|N(e)|).
		\end{equation*}
	It is easy to see that $f_h$ is convex so 
	Jensen's inequality together with 
	\eqref{equation: directed triangle count} implies
		\begin{align}
			\label{equation: transitive subtournaments lower 
			bound}
				N^*_T(h) 
			\geq 
				\binom {n}{2} 
				f_{h-2}{\Bigg ( 
				\frac {N^*_T(3)}{\binom {n}{2}} \Bigg )}
			& = 
				n \cdot \frac{n-1}{2} \cdot f_{h-2}\left(\frac{n-3}{4}\right)
				\nonumber \\ 
			& = 
				n \cdot \frac{n-1}{2} \cdot \prod_{j=0}^{h-3}\left(\frac{n+1}{2^{j+2}}-1\right)
				\nonumber \\ 
			& = 
				\prod_{j=0}^{h-1}\left(\frac{n+1}{2^j}-1\right). 
		\end{align}
	Where in the second equality we used that $\frac{n-3}{4}\ge \frac{2^{h-1}-1-3}{4}=2^{h-3}-1$.
	
	We now show that $Tr_h$ is globally forcing for $h\geq 4$. To see this, suppose 
	that $N^*_T(h) \leq (2^{-\binom {h}{2}} + o(1)) n^{h}$. From \eqref{equation: transitive subtournaments lower bound} 
	this gives $N^*_T(3) \le (1+o(1)) n^3 / 8$. As $h\geq 4$, 
	the function $x^{h-2}$ is strictly convex and by the 
	application of Jensen in \eqref{equation: transitive 
	subtournaments lower bound} we find
	$|N(e)| = (1 \pm o(1))\frac{n}{4}$ for almost all 
	$e \in E(T)$. However, by a result of Chung and Graham (see property $P_4$ and Theorem 1 in \cite{CG}) this property 
	implies that $T$ is quasirandom. Thus $Tr_h$ is globally forcing for $h\geq 4$.

	To complete the proof of the proposition, it remains to show that any 
	non-transitive $h$-vertex tournament $H$ with $h\geq 7$ is not globally forcing.
	To see this, let $V(H) = \{u_1,\ldots ,u_h\}$. For each $n\in {\mathbb N}$ construct an 
	$n$-vertex tournament ${T_1}$ as follows. Let $V(T_1) = V = \{v_1,\ldots , v_n\}$ denote a set of order 
	$n$ and let $V = \cup _{i\in [h]} V_i$ denote a partition with 
	$|V_i| = \lfloor n/h \rfloor $ or $\lceil n/h \rceil $ for all 
	$i\in [h]$. For each edge $\overrightarrow{u_iu_{i'}}$ of $H$, orient all edges from 
	$V_i$ to $V_{i'}$ in $T_1$. Orient the remaining edges of 
	$T_1$ arbitrarily. As each map $\phi : V(H) \to V(T_1)$ 
	with $\phi (u_i) \in V_i$ for all $i\in [h]$ is a labelled embedding, we see that 
		\begin{equation}
			\label{equation: T_1 count}
			N^*_{T_1}(H) \geq \big (h^{-h} - o(1) \big )n^h.
		\end{equation}
	Now, starting from $T_1$, construct a sequence of tournaments $T_1, \ldots , T_{n}$ on vertex 
	set $V(T_1)$, where for $i\geq 1$ each $T_{i+1}$ is obtained from $T_i$ by 
	reorienting all edges to go from $v_i$ to $\{v_{i+1},\ldots , v_{n}\}$. 
	Using that $h^{-h} > 2^{-\binom {h}{2}}$ for $h\geq 7$, by 
	\eqref{equation: T_1 count} there is $\delta >0$ such that 
	$N^*_{T_1}(H) > (2^{-\binom {h}{2}} + \delta ) n^h$ for large $n$. 	
	On the other hand $T_{n} = Tr_n$ and as $H$ is not transitive we have
	$N^*_{T_{n}}(H) = 0$. Since 
	$N^*_{T_{i}}(H) = N^*_{T_{i+1}}(H) \pm hn^{h-1}$, by an 
	intermediate value property 
	$N^*_{T_{i}}(H) = (2^{-\binom {h}{2}} + o(1))n^h$ for some $i\in [n]$. 
	However, $T_i$ is not quasirandom. Indeed as 
	$\big (2^{-\binom {h}{2}} + o(1) \big )n^h = N^*_{T_{i}}(H) \geq (2^{-\binom {h}{2}} + \delta )n^h - ihn^{h-1}$ we have $i \geq \delta n/2h$. 
	As 
	$|d^+_{T_i}(v_j) - d^-_{T_i}(v_j)| = n-2j+1 \geq n/2$ for all 
	$j\in [\min ( i-1, n/4)]$, this contradicts the quasirandom property $\mathcal{P}_5$ 
	from \cite{CG} which requires that $$\sum_{v \in V(T_i)}|d^+_{T_i}(v)-d^-_{T_i}(v)|=o(n^2).$$ Thus $T_i$ is not quasirandom and so $H$ cannot be globally forcing.
\end{proof}
\noindent \textbf{Remark:} Note that our argument shows that in order for $H$ to be globally forcing $N^*_T(H)$ must be asymptotically maximised when $T$ is the random tournament. In other words if we can find a sequence of $n$-vertex tournaments $T_n$ such that $N^*_{T_n}(H) \ge (2^{-\binom{h}{2}}+\delta)n^h$ our argument above (switching edges incident to one vertex at a time) shows that $H$ is not globally forcing. In fact it is enough to find a single such $T_m$ for some $m\ge h$ since if we let $T'_n$ be the blow up of $T_m$ with parts of size $\frac{n}{m}$ then $N^*_{T'_n}(H) \ge (2^{-\binom{h}{2}}+\delta)m^h\left(\frac{n}{m}\right)^h=(2^{-\binom{h}{2}}+\delta)n^h$ so $T'_n$ provides us with our sequence. 

\section{Local forcing and the counting polynomial}
\label{sec: counting polynomial}

To study locally forcing tournaments we introduce the following polynomial.
\begin{dfn}
	\label{dfn: counting poly}
	The \emph{counting polynomial} of an $h$-vertex tournament $H$ is given by
		\begin{equation*}
			p_H(x) 
				:= 
			{\mathbb E}_{\sigma \sim S_h} 
			 \Big ( x^{|F_{\sigma ,H}|} (1-x)^{|B_{\sigma ,H}|}\Big ) = 
			 \frac {1}{h!} \sum _{\sigma  \in S_h} \Big ( x^{|F_{\sigma ,H}|} (1-x)^{|B_{\sigma ,H}|} \Big ).
		\end{equation*}
\end{dfn}  

\vspace{1mm}

\noindent The following lemma collects a number of useful basic facts concerning 
$p_H(x)$.

\begin{lem}
	\label{lemma: properties of counting poly}
	Given an $h$-vertex tournament $H$, the following hold: 
\begin{enumerate}[(i)]
	\item For $\alpha \in [0,1]$ we have 
	${\mathbb E}_{T \sim {\cal T}(n,\alpha )} \big (N^*_T(H) \big ) = 
	(n)_h \times p_H(\alpha )$.
	\item $p_H(1/2) = 2^{-\binom {h}{2}}$ and 
	$p_H(1) = 0$ if $H$ is not transitive.
	\item For all $\alpha \in [0,1/2]$ we have $p_H(1/2+\alpha)=p_H(1/2-\alpha)$. 
\end{enumerate}
\end{lem}

\begin{proof} 
	Let $V(H) = \{u_1,\ldots, u_h\}$. 
	A set $\{{i_1},\ldots ,{i_h}\} \subset [n] = V(T)$ with $i_1< i_2 < \ldots < i_h$ 
	forms a $\sigma $-copy of $H$ if $\overrightarrow{{i_j}{i_k}} \in E(T)$ 
	if and only if $\overrightarrow{u_{\sigma (j)}u_{\sigma (k)}} \in E(H)$. 
	For $T \sim {\cal T}(n,\alpha )$, the probability that a fixed $h$-set forms a $\sigma $-copy 
	of $H$ is $\alpha ^{|F_{\sigma ,H}|}(1-\alpha )^{|B_{\sigma ,H}|}$. Letting 
	$N^*_{T}(H,\sigma )$ denote the number of $\sigma $-copies of $H$ in $T$ gives 
		$${\mathbb E}_{T \sim {\cal T}(n,\alpha )} (N^*_T(H,\sigma ) \big ) 
				= 
		\binom n h \alpha ^{|F_{\sigma ,H}|}(1-\alpha )^{|B_{\sigma ,H}|}.$$ 
	As $N^*_T(H) = \sum _{\sigma \in S_h} N^*_T(H,\sigma )$ we have 
		\begin{equation*}
			 {\mathbb E}_{T \sim {\cal T}(n,\alpha )} 
			 \big ( N^*_T(H) \big ) 
			 	= 
			 \sum _{\sigma \in S_h} 
			 {\mathbb E}_{T \sim {\cal T}(n,\alpha )} \big ( N^*_{T}(H,\sigma ) \big )
			 	= 
			    (n)_h \times p_H(\alpha ).
		\end{equation*} 
	This gives \emph{(i)}.
	
	To see \emph{(ii)}, note that $p_H(1/2) = 2^{-\binom {h}{2}}$ from the definition of $p_H(x)$ since $|F_{\sigma ,H}| + |B_{\sigma ,H}| = \binom h 2$ for every $\sigma \in S_h$. 
	Also note that if $T$ is non-transitive then $B_{\sigma, H}>0$ for all $\sigma$. So, each term of $p_H(x)$ is zero for $x=1$  implying $p_H(1) = 0$.
	
	Lastly, we show \emph{(iii)}. 
	For each $\sigma \in S_h$ let $\overline {\sigma }$ denote 
	the `reversal of $\sigma $', given by $\overline {\sigma } (i) = {\sigma }(h+1-i)$ 
	for all $i\in [h]$. 
	As this gives a bijection from $S_h$ to itself and $F_{\sigma ,H}=B_{\overline{\sigma} ,H}$, for all $\sigma \in S_h$, we have
		\begin{align*}
			p_H(x)	& = \frac {1}{2h!} \sum _{\sigma  \in S_h} \Big ( x^{|F_{\sigma ,H}|} (1-x)^{|B_{\sigma ,H}|}+x^{|F_{\overline{\sigma} ,H}|} (1-x)^{|B_{\overline{\sigma} ,H}|} \Big )\\
				    & = \frac {1}{2h!} \sum _{\sigma  \in S_h} \Big ( x^{|F_{\sigma ,H}|} (1-x)^{|B_{\sigma ,H}|}+x^{|B_{\sigma ,H}|} (1-x)^{|F_{\sigma ,H}|} \Big )
		\end{align*}
	As $x^a(1-x)^b+x^b(1-x)^a$ is symmetric around $1/2$ for any $a,b\in[0,1/2]$ so is $p_H(x)$.
\end{proof}

\noindent We are now ready to give the formal definition of \simple{} which plays a central role in our arguments.

\begin{dfn*}
	A tournament $H$ is said to be \emph{\simple} if $p_H(x) \neq 2^{-\binom{h}{2}}$ 
	for $x \in [0,1] \setminus \{\frac {1}{2}\}$.
\end{dfn*}
\noindent Note that if $H$ is not locally forcing then there is a tournament $T$ which is not quasirandom but has the same local counts of $H$ as the random tournament $\mathcal{T}(n,1/2).$ We will see that not being \simple{} implies that there exists an $\alpha \neq 1/2$ such that $\mathcal{T}(n,\alpha)$ has the same local counts of $H$ as the random tournament $\mathcal{T}(n,1/2)$. Indeed, in the next theorem we prove that any locally forcing tournament must be \simple, which will allow us to pass any restrictions imposed on the tournament by being \simple{} to locally forcing tournaments and in particular prove Theorem \ref{thm: locally-forcing}.

\begin{thm}\label{thm: locally-forcing-is-simple}
	Every locally forcing tournament is \simple. 
\end{thm}

\begin{proof}
	Suppose that $H$ is not \simple. 
	Then by definition ${p_H}(\alpha ) = 2^{-\binom {h}{2}}$ for 
	some $\alpha \in [0,1] \setminus \{1/2\}$. As $p_H(x)$ is symmetric about $1/2$ by 
	Lemma \ref{lemma: properties of counting poly} \emph{(iii)} we can assume 
	$\alpha > 1/2$. Now select $T \sim {\cal T}(n,\alpha )$. 
	For a set $U \subset V(T)$ let $X_U$ denote the 
	random variable $X_U(T) = N^*_{T}(H;U)$. Noting that for $T \sim {\cal T}(n,\alpha )$ 
	we have $T[U] \sim {\cal T}(|U|,\alpha )$, by Lemma \ref{lemma: properties of counting poly} 
	\emph{(i)} we have 
		$${\mathbb E}_{T \sim {\cal T}(n,\alpha )}(X_U) = (|U|)_h \times p_H(\alpha ) = (|U|)_h \times 2^{-\binom {h}{2}}.$$ 
	Since each edge belongs to at most $\binom{h}{2}n^{h-2}$ copies of $H$, $X_U$ is sharply concentrated by Azuma's inequality (see Chapter 7 \cite{AS}) $$\mathbb{P}(|X_U-\mathbb{E}(X_U)|>
	\varepsilon n^h) \le 2e^{\frac{-\varepsilon^2n^{2h}}{2\binom{n}{2}\left(\binom{h}{2}n^{h-2}\right)^2}}=e^{-\Omega(n^2)}.$$ In particular, by a union bound with probability $1-o(1)$ we have 
	$X_U = 2^{-\binom {h}{2}}|U|^h + o(n^h)$ for every $U \subset V(T)$.  
	
	On the other hand the number of forward edges of $T$ is distributed as $\Bin(\binom{n}{2},\alpha)$ so by standard Chernoff type estimates $T$ has $\alpha \binom {n}{2} \pm o(n^2)$ forward edges 
	with probability $1 - o(1)$ (see Appendix A of \cite{AS}). Combined with the previous paragraph, we have shown 
	that there is 
	an $n$-vertex tournament $T$ which satisfies ${\cal P}^*_2(H)$ but not ${\cal P}_1$. 
	Thus $H$ cannot be locally forcing.
\end{proof}

For the rest of this section it will be more convenient to work with a rescaled and recentered counting polynomial $q_H(x) := 2^{\binom {h}{2}} \times p_H\big ( \frac{1+x}{2} \big )$. We now show several restrictions that the \simple{} condition imposes on a tournament. The first one is that a \simple{} tournament should be almost regular. 

\begin{lem}\label{lem: restrict-simple-degree}
    If $H$ is a non-transitive $h$-vertex \simple{} tournament then  
		\begin{equation}
			\label{equation: degree control} 
			\sum _{u \in V(H)} {\big (d^+_H(u) - d^-_H(u) \big )^2} 
					\leq 
			h(h-1).
		\end{equation} 
\end{lem}
\begin{proof}We start by determining the coefficient of $x^2$ in ${q_H}$. We claim it is equal to \begin{align*}\label{eqn:x2-coefficient}
    \frac {1}{6} \times \Big ( \sum _{u \in V(H)} 
	{\big (d^+_H(u) - d^-_H(u) \big )^2} - {h(h-1)} \Big ).\end{align*}
    
    Given $e \in E(H)$ let $I_e$ denote the function 
	$I_e: S_h \to \{-1,1\}$ with 
	$I_e(\sigma ) = 1$  if $e \in F_{\sigma ,H}$ and $I_e(\sigma ) = -1$ if $e \in 
	B_{\sigma ,H}$. Then we have 
		\begin{equation}
			\label{equation: diff_expression_p_H}
			{q_H}(x) 
					= 
			{\mathbb E} _{\sigma \sim S_h} \Big ( (1+x)^{|F_{\sigma ,H}|} (1-x)^{|B_{\sigma, H }|} 
			\Big ) 
					= 
			{\mathbb E} _{\sigma \sim S_h} \Big ( \prod _{e \in E(H)}(1+I_e(\sigma ) x) 
			\Big ).
		\end{equation}
    This implies that the coefficient of $x^2$ in ${q_H}(x)$ is 
		\begin{equation}
			\label{equation: x^2-expression}
			\sum _{\{e,e'\} \in \binom {E(H)}{2}} \Big ( {\mathbb E}_{\sigma \sim S_h} 
			 I_e(\sigma )I_{e'}(\sigma ) \Big ).
		\end{equation}
	The contribution of each pair $\{e,e'\} \in \binom {E(H)}{2}$ to this sum is 
	determined by how these edges meet. If $e \cap e' = \emptyset $ then the contribution 
	to the sum is $0$. If $e$ and $e'$ are both out-edges of a single vertex, 
	say $e = \overrightarrow{u_1u_2}$ and $e' = \overrightarrow {u_1 u_3}$, 
	the contribution to the sum is $1/3$. This is also true if $e$ and $e'$ are both 
	in-edges of a single vertex. Lastly, if $e$ is an out-edge of a vertex $v$ and $e'$ is an in-edge then the contribution to the sum is $-1/3$. By counting 
	the non-zero contributions of pairs $\{e,e'\}$ according to the vertex of $V(H)$ 
	which intersect in, by \eqref{equation: x^2-expression} 
	the coefficient of $x^2$ in ${q_H}(x)$ is 
		\begin{align*}
			\sum _{u \in V(H)} \Big ( \frac{1}{3} \binom {d^+_H(u)}{2}  
			+ \frac {1}{3}\binom {d^-_H(u)}{2} 
			-\frac {1}{3} {d^+_H(u)}{d^-_H(u)} \Big ) 
			& =
			\sum _{u \in V(H)} \frac {(d^+_H(u) - d^-_H(u))^2}{6} 
			- \frac {h(h-1)}{6}.
		\end{align*}
	as claimed.	
	
	Finally, if \eqref{equation: degree control} fails, ${q_H}(\varepsilon ) = 1 + a \varepsilon ^2 + O(\varepsilon ^4)$ with $a > 0$, where we used Lemma 
	\ref{lemma: properties of counting poly} \emph{(ii)} to get ${q_H}(0)=1$ and \emph{(iii)} to get that ${q_H}$ is even. This implies that ${q_H}(\varepsilon ) > 1$ for sufficiently small $\varepsilon >0$. 
	As $H$ is not transitive, Lemma 
	\ref{lemma: properties of counting poly} \emph{(ii)} gives ${q_H}(1) = 0$, 
	and so by the intermediate value theorem ${q_H}(\varepsilon ) = 1$ 
	for some $\varepsilon \in (0,1)$. But this gives $p_H\big ( \frac {1+\varepsilon }{2} \big ) = 2^{-\binom {h}{2}}$ and so $H$ is not \simple.
	\end{proof}

The following lemma shows that, in addition to being almost regular, any \simple{} $H$ must be strongly quasirandom.
\begin{lem}\label{lem:simple-is-quasirandom1}
    If $H$ is a non-transitive $h$-vertex \simple{} tournament then for every ordering $\sigma$ of $H$  $$|F_{\sigma, H}|=\frac{1}{2}\binom{h}{2} \pm h^{3/2}\sqrt{\log h}.$$ 	
\end{lem}
\begin{proof}
    Assume that there is an ordering $\sigma$ such that $f=|F_{\sigma, H}|\ge h^2/4+h^{3/2}\sqrt{\log h}$. If we let $b=|B_{\sigma,H}|$ since $f+b=\binom{h}{2}$ we get $b \le h^2/4- h^{3/2}\sqrt{\log h}$ and $f-b \ge 2h^{3/2}\sqrt{\log h}.$ 
    
    We will show that a single term of $q_H(x)$, $(1+x)^f(1-x)^b>h!$ if we choose $x=\frac34 h^{-1/2}\sqrt{\log h}$. This shows that $q_H(x)>1$ and we can complete the argument as in the previous lemma. To show the above inequality note that
    \begin{align*}
    (1+x)^f(1-x)^b & > e^{(x-x^2)f+(-x-x^2)b}\\
                   & > e^{x(f-b)-x^2h^2/2}\\
                   & > e^{h \log h} =h^h >h!
    \end{align*}
    where in the first inequality we used $1+t>e^{t-t^2}$ for $|t| \le 1/\sqrt{3}$.
\end{proof}

By a result of Spencer \cite{spencer1} it is known that for any tournament $H$ there is an ordering of its vertices $\sigma$ for which $|F_{\sigma, H}| = \frac{1}{2}\binom{h}{2} + \Omega(h^{3/2}).$ So the above result is best possible up to the $\sqrt{\log h}$ factor. Fernandez de la Vega \cite{vega} showed that for the random tournament $H \sim \mathcal{T}(h,1/2)$ w.h.p.\ there is no ordering with $|F_{\sigma, H}|>\frac{1}{2}\binom{h}{2} + 2h^{3/2}.$ So it seems quite likely that the $\sqrt{\log h}$ term in the above lemma is not necessary. In fact for a different definition of quasirandomness we do obtain a result which is best possible up to a constant factor.

\begin{lem}\label{lem:simple-is-quasirandom2}
    If $H$ is a non-transitive $h$-vertex \simple{} tournament then for any disjoint subsets $U,W \subseteq V(T)$ we must have $|e(U,W)|\le \frac{1}{2}|U||W| + 2h^{3/2}$.
\end{lem}
\begin{proof}
    Assume towards a contradiction that there are disjoint sets $U,W\subseteq V(T)$ for which $|e(U,W)| > \frac{1}{2}|U||W| + 2h^{3/2}$. Using these sets we will find many orderings $\sigma$ such that $|F_{\sigma, H}|\ge h^2/4+2h^{3/2}$. Let $S= V(H) \setminus (U \cup W).$ We will always place all vertices of $U$ before all vertices of $W$ in the orderings $\sigma$. Also if $e(S,U \cup W) \ge e(U \cup W, S)$ we place all the vertices of $S$ before all vertices of $U \cup W$ and behind otherwise. Within the sets $U,W,S$ we take the orderings which have more forward edges than backward edges. Note that for any subset of $T$ exactly one out of each of its orderings and its reverse has at least half of the edges going forwards. Therefore, we get at least $\frac{|U|!}{2}\cdot \frac{|W|!}{2}\cdot \frac{|S|!}{2}= \frac{h!}{8\binom{h}{|U|,|W|,|S|}}> h!/(8 \cdot 3^h)$ such orderings (where we used the trinomial expansion). Note that each such ordering $\sigma$ of $H$ has $|F_{\sigma, H}|\ge \frac12 \binom{h}{2}+2h^{3/2}$ since inside each set and between each pair of sets there are at least half of them going forwards and there is a gain of at least $2h^{3/2}$ between $U$ and $W$.

    Assume now $\sigma$ is an ordering of $H$ with $f=|F_{\sigma, H}|\ge h^2/4+2h^{3/2}$. If we let $b=|B_{\sigma,H}|$ similarly as in the previous lemma we obtain that if we choose $x=h^{-1/2}$ then
    \begin{align*}
    (1+x)^f(1-x)^b & > e^{(x-x^2)f+(-x-x^2)b}\\
                   & > e^{x(f-b)-x^2h^2/2}\\
                   & > e^{7h/2} >8 \cdot 3^h.
    \end{align*}

    From the previous two paragraphs we get $q_H(x) >1$ and complete the proof as in Lemma \ref{lem: restrict-simple-degree}.
\end{proof}

Spencer showed in \cite{spencer1} that for some constant $c>0$ in any tournament there are disjoint sets of vertices $U,W$ such that $e(U,W) \ge \frac{1}{2}|U||W|+cn^{3/2}.$ So the above lemma is best possible up to the constant factor.

\begin{proof}[Proof of Theorem \ref{thm: locally-forcing}]
Combining Theorem \ref{thm: locally-forcing-is-simple} with Lemmas \ref{lem: restrict-simple-degree}, \ref{lem:simple-is-quasirandom1} and \ref{lem:simple-is-quasirandom2} proves Theorem \ref{thm: locally-forcing}.
\end{proof}

Lemmas \ref{lem:simple-is-quasirandom1} and \ref{lem:simple-is-quasirandom2} show that in order for $H$ to be \simple{} it needs to be quasirandom in 2 different ways introduced by Chung and Graham in \cite{CG}. But it shows even more, namely that the error term should not only be qualitatively small (as in the definition of quasirandom properties in the introduction) but should in fact be close to their extremal value.

Since the previous two lemmas require $H$ to be quasirandom and are both satisfied for the random tournament $ \mathcal{T}(h,1/2)$ w.h.p.\ a natural guess would be that it should be \simple{} w.h.p..\ This turns out to be false, due to Lemma \ref{lem: restrict-simple-degree}. 

\begin{proof}[Proof of Corollary 
\ref{cor:random-is-not-locally-forcing}]
Let $H \sim \mathcal{T}(h,1/2).$ We show that with positive probability $\sum _{v \in V(H)}(d^+_H(v) - d^-_H(v)\big )^2 > h(h-1)$ which will show the result by Lemma \ref{lem: restrict-simple-degree} through Theorem 
\ref{thm: locally-forcing-is-simple}. 

Let us set $V(H)=[h]$ and define $I_{ij}=1$ if $ij \in E(H)$ and $I_{ij}=-1$ if $ji \in E(H)$, so $\mathbb{P}(I_{ij}=1)=\mathbb{P}(I_{ij}=-1)=1/2$. Note that $I_{ij}=-I_{ji}$ and that otherwise indicators $I_{ij}$ are mutually independent. We have
$$\sum _{i \in [h]}(d^+_H(i) - d^-_H(i)\big )^2 = \sum_{i \in [i]} \left(\sum_{j \in [h], j\neq i}I_{ij} \right)^2=h(h-1)+ 2\sum_{i,j,k \in [h], j \neq k}  I_{ij}I_{ik}. $$

Let $X=\sum_{i,j,k \in [h], j \neq k}  I_{ij}I_{ik}$. Note that $\mathbb{E}(X)=\sum_{i,j,k \in [h], j \neq k}  \mathbb{E}(I_{ij})\mathbb{E}(I_{ik})=0$ since each $I_{ij}$ is independent of $I_{ik}$ when $j\neq k$ and $\mathbb{E}(I_{ij})=0$. 
Our goal is to show $\mathbb{P}(X>0)>c$ for some $c>0$ independent of $h$. To do this we need to determine some higher moments of $X$. Note first that 

$$\mathbb{E}(X^2)=\sum_{i,j,k \in [h], j \neq k}\quad \sum_{i',j',k' \in [h], j' \neq k'}  \mathbb{E}(I_{ij}I_{ik}I_{i'j'}I_{i'k'})=h \binom{h-1}{2} $$

\noindent where we used the fact that unless $i=i',j=j',k=k'$ at least one of the sets $\{i,j\},\{i,k\},\{i',j'\},$ $\{i',k'\}$ is distinct from the others, so its indicator is independent of the others and its contribution vanishes.

It is not hard to estimate $\mathbb{E}X^4$ directly but it requires some case analysis. Instead note that if we replace every occurrence of $I_{ji}$ with $i<j$ with $-I_{ij}$ we get a degree $2$ polynomial whose variables are independent indicators. Thus, by a special case of Bonami-Beckner's hypercontractive inequality (see \cite{odonnell} for a simple proof) we have that $\mathbb{E}(X^4)\le (9\mathbb{E}(X^2))^2=81[\mathbb{E}(X^2)]^2$. 

Now a simple lemma (see \cite{AGK}) says that if a random variable $Y$ has expectation $0$, $\mathbb{E}Y^2>0$ and $\mathbb{E}Y^4/(\mathbb{E}Y^2)^2 \le b$ then $\mathbb{P}(Y>0)>1/(2^{4/3}b)$. So in our case $b=81$ and $\mathbb{P}(X>0) \ge 1/205$.
\end{proof}

\section{Finding \texorpdfstring{\simple{}}{T cliq-forcing} tournaments}
\label{sec: simple tournaments}
In the previous section we saw that in order for a tournament to be \simple{} it needs to be strongly quasirandom. We have also seen that the random tournament is not regular enough to be \simple{} w.h.p.\ It is natural to try to avoid this obstruction by trying next the random regular tournaments. However, the probability space of random regular tournaments is not at all easy to work with so instead we consider a modification of a different probability space on regular tournaments first introduced by Adler, Alon and Ross in \cite{AAR} to study the maximum number of Hamilton paths in tournaments. Using it we show that there are many \simple{} $h$-vertex tournaments for any large enough $h$.

To describe the construction, suppose that triangles $\Delta _1,\ldots , \Delta _L$ and edges $e_1,\ldots, e_F$ form a partition $\mathcal{P}$ of the edge set of the complete graph $K_h$. We now generate a random tournament on the vertex set $[h]$ as follows. Orient the edges of each 
	$\Delta _i$ as a cyclic triangle, each orientation 
	appearing with probability $1/2$ and then orient each $e_j$ uniformly at random as well, so that all the triangles and edges are oriented independently. 
	Write ${\cal D}_{\cal P}$ for the 
	resulting distribution on the set of $h$-vertex tournaments. 

\begin{lem}
	\label{lem: exists-regular-simple-tournaments}
	There exists $c>0$ such that for $h \geq h_0$ and $F \le ch^2$, the random tournament $H \sim {\cal D}_{\cal P}$ is \simple{} w.h.p..
\end{lem}

\begin{proof} 

	Note that $3L+F=\binom{h}{2}$ so by taking $c$ small enough we may assume $L \ge 10F$, implying $h^2/8 \le L\le h^2/6$. 

	To prove the lemma, we again work with $q_H(x)$, showing that 
	with positive probability $q_H(x) \neq 1$ for all $x \in [-1,1]\setminus \{0\}$. Note 
	that in any ordering $\sigma$ of the vertices of $H$, 
	the triangle corresponding to $\Delta _i$ either has 
	two forward edges 
	or two backward edges. Let $J_{i, \sigma }(H) = +1$  
	if the first case and $J_{i,\sigma }(H) = -1$ in the second case. Similarly let $J_{L+j,\sigma}(H)=\pm 1$ depending on whether edge $e_j$ is forwards or backwards. We have
		\begin{align*}
			q_H(x)  
				& = 
			{\mathbb E}_{\sigma \sim S_h}\Big ( 
			(1+ x )^{|F_{\sigma ,H}|}
			(1- x )^{|B_{\sigma ,H}|} \Big )\\
				& =
			{\mathbb E}_{\sigma \sim S_h}\Big ( 
			\prod _{i\in [L]} (1+ x )(1-x) 
			(1+ J_{i,\sigma }(H)x ) \prod _{i\in [L+1,L+F]} (1+ J_{i,\sigma }(H)x )\Big )\\
				& =
			(1-x^2)^L \times 
			{\mathbb E}_{\sigma \sim S_h}\Big ( 
			\prod _{i\in [L+F]} (1+ J_{i,\sigma }(H)x ) \Big )\\
				& = 
			(1-x^2)^{L} \times s_H(x).
		\end{align*}
	As ${q_H}(x)$ is an even polynomial, so 
	is the (random) polynomial $s_H(x)$. Let $s_H(x) 
	= \sum _{\ell\in [(L+F)/2]} c_{2\ell} x^{2\ell}$, where $\{c_{2\ell}\}_{\ell\in [(L+F)/2]}$ are random variables 
	depending on $H$. 
		
	To complete the lemma it will suffice to prove that with high probability  
	$c_{2\ell} \leq \binom {L}{\ell}$ for all $\ell\in [(L+F)/2]$. Indeed, 
	then
		\begin{equation*}
			{s_H}(x) = \sum _{\ell\in [(L+F)/2]} c_{2\ell} x^{2\ell} 
				\leq 
			\sum _{\ell\in [(L+F)/2]} \binom {L}{\ell} x^{2\ell}
			\le
				\sum _{\ell\in [L]} \binom {L}{\ell} x^{2\ell} 
				=
			(1 + x^2)^L,
		\end{equation*}
	which gives ${q_H}(x) = 
	(1-x^2)^L \times s_H(x) \leq 
	(1-x^2)^L \times (1+x^2)^L = 
	(1-x^4)^L < 1$ for $x \in [-1,1]\setminus \{0\}$, i.e. $H$ is \simple. 
	To obtain the required bound on $c_{2\ell}$, note that 
		\begin{equation}
			c_{2\ell} 
					= 
			\sum _{A \in \binom {[L+F]}{2\ell}}
			{\mathbb E} _{\sigma \sim S_h} 
			\Big (  \prod _{i\in A} 
			J_{i,\sigma }(H) \Big ).
		\end{equation}
	Notice that for any $A \in \binom {[L+F]}{2\ell}$ if there is $i\in A$ such that the triangle or edge corresponding to $i$ is vertex disjoint from all other objects indexed by $A$ then ${\mathbb E} _{\sigma \sim S_h} \Big (  \prod_{i\in A} 	J_{i,\sigma }(H) \Big )=0$. Indeed, in this case we can cancel the contribution of a permutation $\sigma $ to the expectation with the permutation $\widetilde \sigma $ obtained from $\sigma $ by reversing the orientation of $\Delta _i$ or $e_{i-L}$. Let $\mathcal{A}\subseteq \binom {[L+F]}{2\ell}$ denote the sets not of this form, i.e. if $A \in {\cal A}$ then every object  indexed by $A$ shares a vertex with another object indexed by $A$. We have shown that
		\begin{equation}
			\label{eqn: c_2l bound}
			c_{2\ell} 
					= 
			\sum _{A \in \mathcal{A}}
			{\mathbb E} _{\sigma \sim S_h} 
			\Big (  \prod _{i\in A} 
			J_{i,\sigma }(H) \Big ).
		\end{equation}
	Note further that any $A \in \mathcal{A}$ must have a subset $A' \subset A$ with $|A'| = \ell$ such that object indexed by $A$ intersects an object intersects a vertex of an object indexed by $A'$. We can choose these $A'$ in $\binom{L+F}{\ell}$ many ways and they span at most $3\ell$ vertices. This leaves us with at most $3\ell h$ options for each of the remaining objects. In particular, we have shown that 
	\begin{equation}
	\label{eqn: bound on A}
	|\mathcal{A}| \le \binom{L+F}{\ell}\cdot \binom{3\ell h}{\ell}.    
	\end{equation}
	 Turning back to upper bounding $c_{2\ell}$ note that \eqref{eqn: c_2l bound} gives us
		\begin{align}
		\label{eqn: sq of coefficient}
			{\mathbb E}_{H \sim {\cal D}_\mathcal{P}} \big ( c_{2\ell} \big )^2 
				& =
			{\mathbb E} _{\sigma , \sigma ' \sim S_h} 
			{\mathbb E}_{H \sim {\cal D}_\mathcal{P}} \Big ( 
			{\sum} _{A,B\in \mathcal{A}}
			\prod _{i\in A} J_{i,\sigma }(H) 
			\prod _{j\in B} J_{j,\sigma '}(H) \Big ).
		\end{align} 
	The contribution given here by all pairs $(A,A)$ in the 
	inner sum is at most $|\mathcal{A}|$. The contribution 
	given by the other pairs vanishes since if there is an $i \in A \setminus B$ then $J_{i,\sigma}(H)$ is independent from all $J_{i',\sigma }(H)$ for $i' \in A \setminus\{i\}$ and all $J_{j,\sigma '}(H)$ for $j \in B$ so
	$${\mathbb E}_{\sigma ,\sigma ' \in S_h} \big (\prod _{i\in A} J_{i,\sigma }(H) 
			\prod _{j\in B} J_{j,\sigma '}(H)\big ) = {\mathbb E}_{\sigma \in S_h}(J_{i,\sigma}) \cdot {\mathbb E_{\sigma, \sigma '\in S_h}} \left (\prod _{i'\in A \setminus\{i\}} J_{i',\sigma }(H) 
			\prod _{j\in B} J_{j,\sigma '}(H)\right )=0,$$ 
	since ${\mathbb E}_{\sigma \in S_h}(J_{i,\sigma})=0$. Thus ${\mathbb E}_{\sigma \in S_h}(c_{2\ell})^2 \le |\mathcal{A}|$  by \eqref{eqn: sq of coefficient}. 
	By Markov's inequality the probability 
	that $c_{2\ell} > \binom {L}{\ell}$ is at most $|\mathcal{A}|/\binom {L}{\ell}^2$. For $\ell \ge 2 \log L$ we get
	\begin{equation}
	    \label{eqn: estimate large l}
	|\mathcal{A}|/\binom {L}{\ell}^2 <\binom {L+F}{2\ell}/\binom {L}{\ell}^2 <\prod_{i=0}^{\ell-1}\left(\frac{L+F-2i}{L-i}\right)^2\binom {2\ell}{\ell}^{-1}\le \left(\frac{L+F}{L}\right)^{2\ell}/\binom {2\ell}{\ell}\le 2^{-\ell}\le \frac{1}{L^2}
	\end{equation}
	\noindent	where we used $L \ge F$ and in the second to last equality we took $c$ small enough and $\ell \geq \ell _0$ since $h \geq h_0$.
	
	When $\ell < 2 \log L$ using \eqref{eqn: bound on A}
	$$|\mathcal{A}|/\binom {L}{\ell}^2 \le \binom{L+F}{\ell}\cdot \binom{3\ell h}{\ell}/\binom {L}{\ell}^2 \le \left(\frac{3\ell h(L+F)}{(L-\ell)^2}\right)^\ell\le \frac{16\log L}{\sqrt{L}}.$$
	Here we used $L \ge 10F$ and $h^2/8 \le L \le h^2/6$ and assumed $h$ is large enough for $\frac{16\log L}{\sqrt{L}} \le 1$.
	Summing over all $\ell$ we have shown that $c_{2\ell} \leq \binom {L}{\ell}$ for all $\ell \in [L]$ 
	with probability at least $1-2\log L  \cdot \frac{16\log L}{\sqrt{L}} +\frac{L-2\log L}{L^2}= 1 - o(1)$.		
\end{proof}
\vspace{1mm}

    \noindent \textbf{Remark: } It is well-known (see \cite{Kirkman}) that for any $h \in {\mathbb N}$ with 
	$h \equiv 1$ or $3 \mod 6$ the complete graph $K_h$ on vertex set 
	$[h]$ admits a Steiner triple decomposition. That is, there is a partition $\mathcal{P}$ of $K_h$ as above for which $F=0$. Note that in this case $\mathcal{D}_\mathcal{P}$ is always a regular tournament. So there is an infinite family of \simple{} regular tournaments.

\section{\texorpdfstring{\esimple{}}{T bip-forcing} tournaments}
\label{sec: reg simple tourns are forcing}

In the previous section we have shown that there are many \simple{} tournaments. Our original goal however was to study locally forcing tournaments and Theorem \ref{thm: locally-forcing-is-simple} only says that locally forcing tournaments are necessarily \simple. While we believe this to be a sufficient condition, our argument requires a weak additional assumption, which we will call \esimple. We show that being 
\esimple{} and \simple{} is equivalent to being locally forcing, which allows us to show that many \simple{} tournaments found in the previous section are in fact locally forcing. We now give the formal definition of \esimple.

\begin{dfn*}
We define the $a$-th order degree counting polynomial of an $h$-vertex tournament by:  $$p_{H,a}(x):=\binom{h}{a}^{-1}2^{-\binom{a}{2}-\binom{h-a}{2}}\sum_{A \in \binom{V(H)}{a}} x^{e(A,V \setminus A)}(1-x)^{e(V \setminus A,A)}.$$
\end{dfn*}

\begin{dfn*}
An $h$-vertex tournament is \textit{\esimple{}} if there is no $\alpha \in (1/2,1]$ such that $p_{H,a}(\alpha)=2^{-\binom{h}{2}}$ for all $1 \le a \le h-1$ simultaneously. 
\end{dfn*}

We have seen that if $H$ fails to be \simple{} then there is an $\alpha \neq 1/2$ such that $\mathcal{T}(n, \alpha )$ has the same local count of $H$ as the random tournament $\mathcal{T}(n,1/2)$. In the following lemma we will see that $H$ not being \esimple{} means that there is an $\alpha \neq 1/2$ such that $\mathcal{T}(n,n,\alpha)$ does have the same count of $H$ as $\mathcal{T}(n,n,1/2)=\mathcal{T}(2n,1/2)$. In fact $H$ not being \esimple{} is a seemingly much stronger restriction on $H$ in the sense that it means that there is an $\alpha>1/2$ such that for \textit{all} $1\le a \le h-1,$ $\mathcal{T}(n,n,\alpha )$ has the correct count of $H$ with $a$ vertices embedded to the left and $h-a$ to the right side for all $a$ simultaneously. 

\begin{lem}\label{lem: locally-forcing-is-extremely-simple}
	Every locally forcing tournament is \esimple. 
\end{lem}

\begin{proof}
	Suppose towards a contradiction that there is an $\alpha \in ( 1/2,1] $ such that $p_{H,a}(\alpha)=2^{-\binom{h}{2}}$ for all $1 \le a \le h-1$. 
	Now select $T \sim {\cal T}(n,n,\alpha )$ with bipartition $(L,R)$. 
	For a set $U \subset L, W \subset R$ let $X_{U,W}$ denote the 
	random variable $X_{U,W}(T) = N^*_{T}(H;U \cup W)$. Probability that a fixed subset of size $a$ of $U$ and a fixed subset of size $h-a$ of $W$ span an embedding of $H$ with $A \subset V(H)$ being embedded to the left is $2^{-\binom{a}{2}-\binom{h-a}{2}}\alpha^{e(A,V \setminus A)}(1-\alpha)^{e(A,v \setminus A)}$. Summing over all possibilities we obtain  
	\begin{align*}
	{\mathbb E}_{T \sim {\cal T}(n,n,\alpha )}(X_{U,W}) & = \sum_{a=0}^h (|U|)_a (|W|)_{h-a} \times \sum_{A \in \binom{V(H)}{a}}2^{-\binom{a}{2}-\binom{h-a}{2}}\alpha^{e(A,V \setminus A)}(1-\alpha)^{e(A,v \setminus A)}\\
	& = (|W|)_h\binom{h}{a}2^{-\binom {h}{2}}+ \sum_{a=1}^{h-1} (|U|)_a (|W|)_{h-a} \times \binom{h}{a}p_{H,a}(\alpha)+ (|U|)_{h}\binom{h}{a}2^{-\binom {h}{2}}\\
	& = \sum_{a=0}^h \binom{h}{a}(|U|)_a (|W|)_{h-a}\times 2^{-\binom {h}{2}} \\
	& =  (|U|+|W|)_h \times 2^{-\binom {h}{2}},
	\end{align*} 
	Where in the last equality we used the identity $\sum_{a=0}^h \binom{b}{a}\binom{c}{h-a}=\binom{b+c}{h}$ holding for any $b,c$.
	
	The rest of the proof proceeds in the same way as the proof of Theorem \ref{thm: locally-forcing-is-simple} except that the number of forwards edges of $T(n,n,\alpha)$ is $(1/4+\alpha/2)\binom {2n}{2}\pm o(n^2)$ with probability $1 - o(1).$	
\end{proof}

This lemma together with Theorem \ref{thm: locally-forcing-is-simple} shows the `only if' part of Theorem \ref{thm:-simple and esimple iff locally forcing}. We postpone the proof of the `if' part to the next section, showing first that there are many tournaments which are \esimple{}, in addition to being \simple.

As already mentioned, we believe being \esimple{} is a very weak additional condition which might be already implied by being \simple. It is even possible that all tournaments are \esimple. We now give certain simple conditions that make a tournament \esimple. We say that an $h$-vertex tournament $H$ is \textit{nearly regular} if $|d^+_H(v)-d^-_H(v)|< \sqrt{h}/2$ for all $v \in V(H)$.

\begin{lem}\label{lem: nearly regular is esimple}
    Any nearly regular tournament is \esimple.
\end{lem}
\begin{proof}
We are going to show that $p_{H,1}(x)+p_{H,h-1}(x)<2\cdot 2^{-\binom{h}{2}}$ for all $x \in ( 1/2, 1]$. It will be easier to work with the following polynomials 
$$q_{H,a}(x)=2^{\binom{h}{2}}\cdot p_{H,a}\left (\frac{1+x}{2}\right)=\binom{h}{a}^{-1} \sum_{A \in \binom{V(H)}{a}}(1+x)^{e(A,V \setminus A)}(1-x)^{e(V \setminus A,A)}.$$
So that
\begin{align*}
    q_{H,1}(x)+q_{H,h-1}(x) & = \frac{1}{h}\sum_{v \in V(H)}(1+x)^{d^+(v)}(1-x)^{d^-(v)}+(1+x)^{d^-(v)}(1-x)^{d^+(v)}\\
    & = \frac{1}{h}\sum_{v \in V(H)}(1-x^2)^{\min(d^+(v),d^-(v))}\left((1+x)^{|d^+(v)-d^-(v)|}+(1-x)^{|d^-(v)-d^+(v)|}\right) \\
    & \le (1-x^2)^{h/6}\left((1+x)^{\floor{\sqrt{h}/2}}+(1-x)^{\floor{\sqrt{h}/2}}\right)
\end{align*}
Where we used that $\min(d^+(v),d^-(v)) \ge (h-1)/2-\sqrt{h}/2\ge h/6$, when $h\ge 5$ and that if $h < 5$ near regularity implies regularity so $(h-1)/2\ge h/6$. Note that since $\binom{\sqrt{h}/2}{2\ell} \le \frac{(\sqrt{h}/2)^{2\ell}}{(2\ell)!}\le \left(\frac{h}{6\ell}\right)^\ell\le \binom{h/6}{\ell}$
$$(1+x)^{\floor{\sqrt{h}/2}}+(1-x)^{\floor{\sqrt{h}/2}}\le 2\sum_{\ell=0}^{\sqrt{h}/4} \binom{\sqrt{h}/2}{2\ell}x^{2\ell }\le 2\sum_{\ell=0}^{\sqrt{h}/4} \binom{h/6}{\ell}x^{2\ell }\le 2(1+x^2)^{h/6}.$$
Combining the above two inequalities we obtain $q_{H,1}(x)+q_{H,h-1}(x)\le 2(1-x^4)^{h/6} < 2$ as desired.
\end{proof}

Note that we know by Lemma \ref{lem: restrict-simple-degree} that any \simple{} tournament must be almost regular but the restriction of the above lemma is slightly stronger. On the other hand our argument only uses a much weaker property than the one given to us by the definition of being \esimple.\vspace{1mm} 

\noindent \textbf{Remark:} It is not hard to adapt our proof of Lemma \ref{lem: exists-regular-simple-tournaments} to show that the random tournament $\mathcal{D}_{\mathcal{P}}$ is \esimple{} provided $\mathcal{P}$ consists of at most $ch^2$ edges (and the remaining objects are triangles) for sufficiently small $c$. So in some sense all our examples from the previous section are in fact locally forcing.\vspace{1mm} 

To conclude the section we deduce Theorem \ref{thm: exist-locally-forcing}, assuming the `if' statement from Theorem \ref{thm:-simple and esimple iff locally forcing} (which will be proven in the next section).

\begin{proof}[Proof of Theorem \ref{thm: exist-locally-forcing}]
Let $\mathcal{P}$ be a partition of $K_h$ consisting of triangles and edges with every vertex incident to less than $\sqrt{h}/2$ of the edges in $\mathcal{P}$. It is easy to see that such a partition exists for any large enough $h$ (for example by a result of Gustavsson \cite{gustavsson}). 

By Lemma \ref{lem: exists-regular-simple-tournaments} $H \sim \mathcal{D}_{\mathcal{P}}$ is \simple{} w.h.p.. Furthermore, any $H \sim \mathcal{D}_{\mathcal{P}}$ is nearly regular and so by Lemma \ref{lem: nearly regular is esimple} it is \esimple{}. Putting these together, Theorem \ref{thm:-simple and esimple iff locally forcing} shows that the tournament $H \sim \mathcal{D}_{\mathcal{P}}$ is w.h.p.\ locally forcing.
\end{proof}

\section{Proving local forcing}
\label{sec: proof of simple and esimple implies local forcing}

Before proceeding to the proof of the `if' statement of Theorem \ref{thm:-simple and esimple iff locally forcing} in subsection \ref{subsec: simple reg}, we first recall some results on regularity lemmas for directed graphs in the next subsection.

\subsection{Regularity and counting lemmas for directed graphs}

A directed graph $D = (V,E)$ consists of a set $V$ of vertices and a set of edges $E \subset V \times V$. Clearly any tournament is also a directed graph. Given disjoint sets $A,B \subset V$ we write $E(A,B)$ to denote the collection of edges $(a,b)\in E \cap (A \times B)$ and $e(A,B) = |E(A,B)|$. We will write $d(A,B)$ to denote the \emph{density} of the pair $(A,B)$, given by 
	\begin{equation*}
	d(A,B) = \frac {|E(A,B)|}{|A||B|}.
	\end{equation*}
Note that if $T$ is a tournament then $d(A,B) = 1 - d(B,A)$. Given disjoint sets $X,Y \subset V$ we say that $(X,Y)$ is an $\varepsilon $-regular pair if all $X' \subset X$ and $Y' \subset Y$ with $|X'| \geq \varepsilon |X|$ and $|Y'| \geq \varepsilon |Y|$ satisfy 
	\begin{equation*}
		|d(X',Y') - d(X,Y)| \leq \varepsilon \qquad \mbox{and} \qquad 
		|d(Y',X') - d(Y,X)| \leq \varepsilon.
	\end{equation*}
A partition of $V = \{V_0,V_1,\ldots ,V_K\}$ is said to be an $\varepsilon $-regular partition of $D$ if: 
	\begin{enumerate}[(i)]
		\item $|V_0| \leq \varepsilon |V|$, 
		\item $|V_1| = \cdots = |V_K|$,
		\item all but at most $\varepsilon \binom {K}{2}$ of the pairs $(V_i,V_j)$ with $1\leq i < j \leq K$ are $\varepsilon $-regular. 
	\end{enumerate}
The set $V_0$ is called the exceptional set and the sets $V_1,\ldots ,V_K$ are clusters. This partition is said to refine a partition $V = U_1\cup \cdots \cup U_L$ if for all $k\in [K]$ we have $V_k \subset U_l$ for some $l\in [L]$.
\vspace{2mm}

The directed regularity lemma of Alon and Shapira from \cite{ASh}, which extends Szemer\'edi's graph regularity lemma \cite{Sz},  states the following:

\begin{thm}
	\label{thm: digraph regularity}
	Given $m,L \in {\mathbb N}$ and $ \varepsilon > 0$, there is $M = M(m,L,\varepsilon )$ with the following property. 
	Given a directed graph $D = (V,E)$ with $|V| \geq M$ and a partition 
	$V = U_1\cup \cdots \cup U_L$ there is an $\varepsilon $-regular 
	partition $\{V_0,V_1,\ldots, V_K\}$ with $m \leq K \leq M$, which refines $U_1\cup \cdots \cup U_L$.
\end{thm}

\noindent \textbf{Remark:} While the theorem in \cite{ASh} does not mention 
refinements, it follows easily from the proof.\vspace{2mm}

A convenient structure associated with a regular partition $\{V_0,V_1,\ldots ,V_K\}$ 
is the reduced graph ${\cal R}$, which has vertex set $\{V_1,\ldots ,V_K\}$ with the property that $V_{i}V_j$ is an edge of ${\cal R}$ if $(V_i,V_j)$ is an $\varepsilon $-regular pair. Note that by definition ${\cal R}$ has at least $(1-\varepsilon )\binom {K}{2}$ edges.
\vspace{1mm}

We will also require the following counting lemma. 

\begin{lem}
	\label{lem: counting lemma}
	Let $T = (V,E)$ be a tournament and $V_1,\ldots ,V_h$ be disjoint subsets of $V$. 
	Suppose that for each $1\leq i < j \leq h$ the pair $(V_i,V_j)$ is 
	$\varepsilon $-regular with density $d_{ij}$, with $d_{ji} = 1 - d_{ij}$. 
	Then given an $h$-vertex tournament $H$ with $V(H) = \{u_1,\ldots ,u_h\}$, 
	the number of copies of $H$ in $V$, with $u_i$ sent to $V_i$ for all 
	$i\in [h]$ is
		\begin{equation*}
			\Big ( \prod _{{\overrightarrow {u_iu_j}} \in E(H)} d_{ij} 
			\pm C_h \varepsilon \mbox{ } \Big ) \prod _{l \in [h]} |V_{l}|.
		\end{equation*}
\end{lem}

\begin{proof}
	By deleting directed edges of $T$ which are not 
	of the form $\overrightarrow{v_iv_j}$ with $v_i \in V_i$, $v_j \in V_j$ 
	and $\overrightarrow {u_iu_j} \in E(H)$ and ignoring the directions 
	of the remaining edges, this follows immediately from the usual 
	graph counting lemma; see \cite{schacht}.
\end{proof}

An embedding $\phi $ of a $h$-vertex tournament $H$ into a tournament $T$ is said to be 
partite with respect to the disjoint sets $U_1,\ldots ,U_h \subset V(T)$ if each set 
$U_i$ receives one vertex of the embedding. Let Emb$_T(H;U_1,\ldots ,U_h)$ denote 
the set of all such $\phi $ and let $N^*_T(H; U_1,\ldots ,U_h) = |\mbox{Emb}_T(H;U_1,\ldots ,U_h)|$.

The following proposition gives a `partite version' of property ${\cal P}_2^*(H)$.

\begin{prop}
	\label{prop: partite_counting}
	Let $H$ be an $h$-vertex tournament. Suppose that $T$ is an 
	$n$-vertex tournament with $N^*_T(H; U) = \rho |U|^h \pm C$ for 
	all $U \subset V(T)$. Then $N^*_T(H;U_1,\ldots ,U_h) = 
	h!\rho\prod _{i\in [h]} |U_i| \pm 2^hC$ for all disjoint sets 
	$U_1,\ldots ,U_h \subset V(T)$.
\end{prop}

\begin{proof}
	By the inclusion-exclusion principle we have
		\begin{align*}
			N^*_T(H;U_1,\ldots ,U_h) 
				& =
			\sum _{r=1}^h 
			(-1)^{h-r} \Big ( \sum _{I \in \binom {[h]}{r}} 
			 N^*_T(H; \bigcup _{i\in I} U_i ) \Big ).
		\end{align*}
	Using $N^*_T(H; \bigcup _{i\in I} U_i) = 
	\rho\big ( \sum _{i\in I} |U_i| \big ) ^h \pm C$ 
	for all $I \subset [h]$ gives
		\begin{align*}
			N^*_T(H;U_1,\ldots ,U_h) 
			 	& =
			 \rho  \times \bigg ( \sum _{r=1}^h 
			(-1)^{h-r} \Big ( \sum _{I \in \binom {[h]}{r}} 
			\big ( \sum _{i\in I} |U_i| \big ) ^h \Big ) \bigg ) \pm 
			 2^h C
			 	= 
			 \rho \times \big ( h! \prod _{i\in [h]} |U_i| \big ) \pm 
			 2^h C .
		\end{align*}
	The final equality holds as the summed term in the penultimate equation 
	can be viewed as counting, using inclusion-exclusion, the maps 
	$g: [h] \to \bigcup _{i\in [h]} U_i$ sending each $i\in [h]$ to distinct $U_j$. Indeed there are $h! \prod _{i\in [h]} |U_i|$ such $g$ which intersect each of the $h$ sets $U_i$ and for each $1 \le r \le h$ there are $\sum _{I \in \binom {[h]}{r}} \big ( \sum _{i\in I} |U_i| \big ) ^h$ such maps which intersect at most $r$ of the $U_i$'s. 
	\end{proof}

\subsection{\texorpdfstring{\simple{}}{T cliq-forcing} and \texorpdfstring{\esimple{}}{T bip-forcing} imply local forcing}
\label{subsec: simple reg}

Before proceeding to the main result of this section, we note a simple consequence of Ramsey's theorem.

	\begin{lem}
		\label{lem: Turan-Ramsey decomp into cliques}
	Given $\alpha >0$ and $k, \ell \in {\mathbb N}$ there is 
	$\gamma = \gamma (k,\ell, \alpha )> 0 $ and $n_0 = 
	n_0(k,\ell, \alpha )\in {\mathbb N}$ with 
	the following property. Suppose that $G$ is an $n$-vertex graph with 
	$n \geq n_0$ and at least $(1-\gamma )\binom {n}{2}$ edges. 
	Then in any $k$-colouring 
	of $E(G)$ there are vertex disjoint sets $U_1,\ldots , U_M 
	\subset V(G)$ so that: 
		\begin{enumerate}[(i)]
			\item $G[U_m]$ is a monochromatic clique of order $\ell $ 
			for all $m\in [M]$;
			\item $|V(G) \setminus \big ( \cup _{m \in [M]} U_m \big ) | 
			\leq \alpha n$.
		\end{enumerate}
	\end{lem}
	
	\begin{proof}
		Set $R:= R_k({\ell })$, the $k$-colour Ramsey number 
		of an ${\ell }$-vertex clique. Also set $\gamma = {\alpha }^2/2R$ and 
		$n_0 = \lceil 2R / \alpha \rceil $. Suppose 
		we are given a $k$-colouring of $E(G)$ as in the statement. 
		To prove the lemma it suffices to show that every 
		$W \subset V(G)$ 
		with $|W| \geq \alpha n$ contains a set $U \subset W$ with 
		$|U| = \ell$ so that $G[U]$ is a monochromatic clique. Indeed, 
		using this property we can then greedily find sets $U_1,\ldots ,U_M$ 
		as in the lemma.
		
		To see that this holds, first note that since $|W| \geq \alpha n \geq 2R$ and $|W|^2 \geq \alpha ^2n^2 = 2\gamma Rn^2$, we have 
			\begin{equation*}
				e(G[W]) 
					\geq 
				\binom {|W|}{2} - \gamma \binom {n}{2} 
					> 
				\Big ( 1 - \frac {1}{R} \Big ) \frac {|W|^2}{2}								
				+ \Big ( \frac {|W|^2}{2R} - \frac {|W|}{2} - \frac {\gamma n^2}{2} \Big )
					\geq
				\Big ( 1 - \frac {1}{R} \Big ) \frac {|W|^2}{2}.
			\end{equation*}
		By Tur\'an's theorem there is $W' \subset W$ 
		such that $|W'| = R$ and $G[W']$ is complete. 
		As $G[W]$ is $k$-coloured, from Ramsey's theorem and the 
		definition of $R$,
		there is $U \subset W'$ with $|U| = \ell $ and $G[U]$ is 
		monochromatic. This completes the proof.
	\end{proof}

We now turn to the proof of the if part of Theorem \ref{thm:-simple and esimple iff locally forcing}.

\begin{thm}
Any \simple{} and \esimple{} tournament is locally forcing.
\end{thm}
\begin{proof} 
Let $H$ be a \simple{} and \esimple{} tournament with $|H| = h$. We are required to show that given $\theta >0$ there is $\delta > 0$ and $n_0$ such that the following holds. Suppose that $T$ is an $n$-vertex tournament with $n \geq n_0$ which satisfies
\begin{equation}
	\label{equation: correct H count}
	N^*_T(H;U) = 2^{-\binom {h}{2}}|U|^h \pm \delta n^h
\end{equation} 
for all $U \subset V(T)$. Then any ordering of $V(T)$ has at most $\frac {1}{2}\binom {n}{2} + \theta n^2$ forward edges. \vspace{1mm}

Let us now give a sketch of the proof before delving into the details. Let $v_1,\ldots , v_n$ be an ordering of $V(T)$. We begin by splitting the vertices into consecutive sets $U_{\ell}$ of almost equal size. Then we apply the regularity lemma (Theorem \ref{thm: digraph regularity}) to refine this partition. %The benefit of refining is that for any 2 clusters belonging to distinct $U_\ell$'s all elements of one always appear before the other. 
Taking the reduced graph we define $\mathcal{R}_\ell$ to be its subgraph consisting of clusters contained in $U_\ell$. We colour all edges of $\mathcal{R}_\ell$ which join clusters with density between them belonging to the same small interval in the colour indexed by this interval. If we split $[0,1]$ into finitely many such intervals we obtain a colouring of $\mathcal{R}_\ell$ to which we can apply Lemma \ref{lem: Turan-Ramsey decomp into cliques} to group most of the clusters inside each $\mathcal{R}_\ell$ into monochromatic cliques. We now show that all edges inside these cliques must have density close to $1/2$. To see this assume the opposite, so that there is a clique $C$ with all edges having density bounded away from $1/2$. Now using Lemma \ref{lem: counting lemma} and $H$ being \simple{} we conclude that there are too few copies of $H$ between the clusters of $C$, compared to what is guaranteed by Proposition \ref{prop: partite_counting}, which holds by \eqref{equation: correct H count}.

We then proceed to upper bound the number of forwards edges of $T$. The main contribution comes from edges between $\varepsilon$-regular pairs of clusters between different $U_\ell$'s. 
To bound this number for a pair of cliques belonging to different $U_\ell$'s and a fixed $d>1/2+\theta$ we build an auxiliary bipartite graph with clusters of the two cliques making the sides of the bipartition and making an edge for any pair of $\varepsilon$-regular clusters which have density roughly $d$ in the forwards direction. We show that this auxiliary graph can not contain $K_{a,h-a}$ for some $a$, as otherwise by using a similar argument as before we get $h$ clusters between which we have a wrong count of the copies of $H$ using the fact $H$ is \esimple. By grouping densities and applying the above reasoning for each group we show that there are few forwards edges between the two cliques. Trivially bounding the remaining contributions we show that there are fewer forwards edges than required, completing the proof.

Before beginning we will fix a number of parameters to be used in the proof. %By decreasing $\theta $ we may assume that $\theta ^2 \leq 2^{-\binom {h}{2}}(4h)^{-1}/4$. 
Let $\xi=\xi(H,\theta)$ be the minimum of the continuous function $f(x)=\max_{a \in [h-1]}(|2^{\binom{h}{2}}p_{H,a}(x)-1|)$ on the interval $[1/2+\theta,1]$. Note that since $H$ is \esimple{} we have $f(x)>0$ for each $x$ in this range so since $f$ is continuous we get $\xi>0$. Set $\eta = \xi 2^{-\binom {h}{2} -2}h^{-2}$. Since $H$ is \simple{} and $p_H(x)$ is continuous there is $\zeta \in (0, \eta )$ with the property that if $x \in [0,1]$ and $p_H(x) \geq 2^{-\binom {h}{2}} - \zeta $ then $x = (1 \pm \eta)/2$. Take $\alpha = \theta /64$, $m =\max ( \lceil 4h^2/\zeta \rceil , \lceil 2/\eta \rceil ),$ $L = \lceil 16\theta ^{-1} \rceil$ and $N_1 = \lceil (2(\theta \eta )^{-1} h)^h\rceil $. Also set $m_{min} = 4Ln_0(m,N_1,\gamma )$ and $\gamma = \gamma (m,N_1,\alpha )$ as in Lemma \ref{lem: Turan-Ramsey decomp into cliques}. With $C_h$ as in Lemma \ref{lem: counting lemma}, take $\varepsilon >0$ so that 
	\begin{equation*}
		\varepsilon 
			= 
		\min \Big ( \frac{1}{4L} , \frac {\gamma }{8L^2}, 
		\frac {\zeta }{4C_h}, \frac{\theta ^3}{32} \Big ).
	\end{equation*}
Lastly, set $n_0 = M = \max(M(m_{min}, L, \varepsilon ),8L)$ as in Theorem \ref{thm: digraph regularity} and $\delta = \zeta (4M)^{-h} /2$.
\vspace{1.5mm}

To begin the proof set $U_{\ell } = \{v_i \in V(T): i\in [(\ell -1)n/L, \ell n/L)\}$ for all $\ell \in [L]$.  Note that $|U_\ell| \ge n/L-2$. Provided $n\geq n_0$, we may apply the Theorem \ref{thm: digraph regularity} to $T$ to obtain an $\varepsilon $-regular partition $\{V_{k }\}_{{k} \in [K]} \cup \{V_0\}$ refining $\{U_{\ell }\}_{\ell \in [L]}$, with $m_{min} \leq K \leq M$. Let ${\cal R}$ denote the reduced graph of this partition and ${\cal R}_{\ell }$ denote the subgraph of ${\cal R}$ consisting of the clusters contained in $U_{\ell }$. Setting $W_{\ell }=|{\cal R}_{\ell }|$, we have $W_{\ell } \geq \frac{K}{2L}$ for all $\ell \in [L]$. Indeed, since $|V_{k}| = |V_{k'}|$ for all $k,k' \in [K]$ we have 
	\begin{equation*}
		n 
			\geq 
		\sum _{k \in [K]} |V_{k}| 
			= 
		\frac {K}{W_{\ell}} \sum _{V_{k} \in V({\cal R}_{\ell })} |V_{k }| 
			\geq 
		\frac {K}{W_{\ell }} \Big ( |U_{\ell }| - |V_0| \Big ) 
			\geq 
		\frac {K}{W_{\ell }} \times \frac {n}{2L}, 
\end{equation*}
using $|U_{\ell }| - |V_0| \geq (n/L - 2) -\varepsilon n \geq \frac{n}{2L}$. Rearranging, we find $W_{\ell } \geq \frac{K}{2L}$ for all $\ell \in [L]$. \vspace{3mm}

\noindent \textbf{Claim:} Each ${\cal R}_{\ell}$ contains a collection of vertex disjoint cliques ${\cal C}_{\ell }$ with the following properties: 
	\begin{enumerate}[(i)]
		\item Each clique $C \in {\cal C}_{\ell }$ has order $N_1$,
		\item $|{\cal R}_{\ell } \setminus ( \cup _{C \in {\cal C}_{\ell }} C )| \leq \alpha W_{\ell }$,
		\item $d(V_{k},V_{k '}) = (1/2 \pm \eta )|V_{k}|
		|V_{k '}|$ 
		for each edge $V_{k}V_{k '}$ in a clique $C \in {\cal C}_{\ell }$.
	\end{enumerate}
\vspace{3.5mm} 

To prove the claim, colour the edges of ${\cal R}_{\ell }$ 
with $m$ colours, where each pair $V_kV_{k'}$ with $k < k'$ gets color $j\in [m-1]$ if $d(V_{k}, V_{k'}) \in j/m \pm 1/m$ (ties broken arbitrarily). The graph ${\cal R}_{\ell }$ contains at least 
$\binom {W_{\ell }}{2} - \varepsilon \binom {K}{2} > (1 - \gamma ) \binom {W_{\ell }}{2}$ edges, since $W_{\ell } \geq \frac{K}{2L}$, $\frac{\gamma}{8L^2} \geq \varepsilon$ and $K \ge m_{min} \ge 4L$. Therefore, from Lemma \ref{lem: Turan-Ramsey decomp into cliques}, since $|{\cal R}_{\ell }| \geq \frac{K}{2L} \geq \frac{m_{min}}{2L} \geq n_0(m, N_1,\alpha )$, the graph ${\cal R}_{\ell }$ contains a collection ${\cal C}_{\ell }$ of vertex disjoint monochromatic cliques which satisfy parts (i) and (ii) from the claim.\vspace{1mm}

It remains to show that part (iii) holds. Let $C \in {\cal C}_{\ell }$ be monochromatic with colour $j$ and $V_{k_1},\ldots ,V_{k_h} \in C$ with $k_1< ...< k_h$. As each pair $(V_{k}, V_{k'})$ in $C$ is $\varepsilon $-regular with $d(V_{k_i},V_{k_{i'}}) = (j\pm 1)/m$, by Lemma \ref{lem: counting lemma} we have
	\begin{align*}
		N^*_T(H; V_{{k}_1}, V_{{k}_2},\ldots ,V_{{k}_h}) 
			&= 
		\sum _{\sigma \in S_h} 
		\bigg ( \Big (\frac {j \pm 1}{m}\Big ) ^{|F_{\sigma ,H}|} 
		\Big (\frac {m-j \pm 1}{m} \Big )^{|B_{\sigma ,H}|} \pm C_h\varepsilon 
		\bigg ) \prod _{i \in [h]} |V_{k_{i}}|\\ 
			& = 
		h!\big ( p_H(j/m) \pm (C_h \varepsilon+ h^2m^{-1}) \big ) 
		\prod _{i \in [h]} |V_{k_{i}}| \\
			&= 
		h!\big ( p_H(j /m) \pm \zeta /2 \big ) 
		\prod _{i \in [h]} |V_{k_{i}}|.
	\end{align*} 
In the second equality here we used repeatedly the fact that for any $x,y,t \in [0,1]$ such that $x+t \le 1$ we have $(x+t)(y+t) \le (x+t)y+t \le xy+2t$. 
On the other hand, from \eqref{equation: correct H count} and Proposition \ref{prop: partite_counting} we also have
	\begin{equation}
		\label{equation: partite bound}
		N^*_T(H; V_{k_1}, V_{k_2},\ldots ,V_{k_h}) 
			=  
		h!2^{-\binom {h}{2}} \prod _{i \in [h]} |V_{{k}_{i}}| \pm 2^h\delta n^h 
			=  
		h! \big ( 2^{-\binom {h}{2}}  \pm \zeta /2 \big ) 
		\prod _{i \in [h]} |V_{k _{i}}|.
	\end{equation}
Here we have used that all clusters $V_{k}$ satisfy $|V_{k}| \geq (n-|V_0|)/K \geq n/2M$ and $\delta (4M)^h = \zeta /2$. Combined, these bounds give 
$p_H(j/m) = 2^{-\binom {h}{2}} \pm \zeta$. By our choice of $\zeta $ this forces $j/m = (1 \pm \eta )/2 $ and so $d(V_{k}, V_{{k}'}) = (1\pm \eta )/2 \pm m^{-1} = 1/2 \pm \eta $ for any pair $V_kV_{k'}$ contained in a clique $C \in {\cal C}_{\ell }$ giving (iii). This completes the proof of the claim. \vspace{3mm}

We can now proceed to prove an upper bound on the number of forward edges of $T$. To do this, first fix $1 \leq \ell < \ell ' \leq L$ and $C \in {\cal C}_{\ell }$ and $C' \in \mathcal{C}_{\ell'}$. We will prove that 
	\begin{equation}
		\label{equation: clique control}
		\sum _{\substack{V_k \in C,V_{k'}\in C':\\ (V_k,V_{k'}) \mbox{ } \varepsilon -\mbox{reg}}} |E(V_k, V_{k'})| 
		\leq \Big (\frac {1}{2} + \frac{3 \theta }{2} \Big )\Big | \bigcup _{V_k \in C} V_k \Big | \Big | \bigcup _{V_{k'} \in C'} V_k \Big |.
	\end{equation}
To see this fix some $d > 1/2 + \theta $ and consider the auxilliary bipartite graph $G$ with vertex set on one side being $C$ and on the other side $C'$. We put an edge between clusters $V_k \in C$ and $V_{k'} \in C'$ if the pair $(V_k,V_{k'})$ is $\varepsilon$-regular and has density $d(V_{k}, V_{k'}) = d \pm \eta$. Let $a$ be such that $|2^{\binom{h}{2}}p_{H,a}(d)-1|\ge  \xi$, which exists since $d> 1/2+\theta$. Our goal is to show that $G$ is $K_{a,h-a}$-free, which by K\H{o}v\'ari-S\'os-Tur\'an theorem is going to tell us that $G$ is sparse, allowing us to bound the number of forwards edges of $T$ between $C$ and $C'$. 

To see this assume towards a contradiction that $G$ contains a $K_{a,h-a}$. Let $V_{k_1},\ldots, V_{k_a}$ make one side and $V_{k_{a+1}},\ldots, V_{k_h}$ the other of this $K_{a,h-a}$. By part (iii) of the 
claim, all pairs $(V_{k_i},V_{k_{i'}})$ are $\varepsilon $-regular with $d(V_{{k}_i},V_{k _{i'}}) = 1/2 \pm \eta $ for distinct $i,i'\in [a]$ or distinct $i,i' \in [a+1,h]$. Any $\phi \in \mbox{Emb}_T(H;V_{k_1},\ldots ,V_{k_h})$ embeds $\binom{a}{2}+\binom {h-a}{2}$ edges of $H$ into pairs with density $1/2 \pm \eta $. If $A:=\phi^{-1}(\{V_{k_1},\ldots, V_{k_a}\})$ then $e(A,V\setminus A)$ edges of $H$ get embedded into pairs with density $d \pm \eta $, and $e(V \setminus A,A)$ edges into pairs with density $1- d \pm \eta$. So Lemma \ref{lem: counting lemma} gives
{\footnotesize
	\begin{align}
		\frac{N^*_T(H; V_{k_1},\ldots ,V_{k_{h}}) }{\prod _{i \in [h]} |V_{{k}_i}|}
			& = 
		a!(h-a)!\sum_{A \in \binom{V(H)}{a}} \left (  \left (\frac12 \pm \eta  \right)^{\binom {a}{2}+\binom {h-a}{2}} 
		 ( d \pm \eta  )^{e(A,V\setminus A)} 
		 ( 1 - d \pm \eta  )^{e(V\setminus A,A)} 
		\pm C_h \varepsilon  \right)
		\nonumber \\
			& = 
		h! \binom{h}{a}^{-1}
		\sum_{A \in \binom{V(H)}{a}}\left( 2^{-\binom{a}{2}-\binom{h-a}{2}}d ^{e(A,V\setminus A)}(1-d)^{e(V\setminus A,A)} 
		\pm (h^2\eta + C_h \varepsilon ) \right)\nonumber \\
		&= 
		h! (p_{H,a}(d)
		\pm (h^2\eta + C_h \varepsilon )  )
		. \label{eqn: homom count}
	\end{align}	
}
As $(h^2 \eta + C_h \varepsilon )+ \zeta \leq 
h^2 \eta + 2\zeta \leq 3h^2 \eta < \xi 2^{-\binom {h}{2}}$ this gives
\begin{align*}
		\left| \frac {N^*_T(H; V_{k_1},\ldots ,V_{k_{h}}) }
		{\prod _{i \in [h]} |V_{{k}_i}|h!}-2^{-\binom {h}{2}}\right|
			& \geq \zeta,
	\end{align*}
which contradicts \eqref{equation: partite bound}. Thus, there is no $K_{a,h-a}$ in $G$. The K\H{o}v\'ari-S\'os-Tur\'an theorem (see \cite{kst}) now tells us that $G$ has at most $hN_1^{2-1/h}$ edges (recall that $|C|=|C'|=N_1$.) Since $d$ was arbitrary (provided it is bigger than $1/2+\theta$) by splitting the interval $[1/2+\theta,1]$ into at most $\eta^{-1}$ subintervals of width at least $2\eta$ and building a graph as above for each of these subintervals with $d$ being equal to the center of the interval we obtain that there are at most $\eta ^{-1}hN_1^{2-1/h}$ pairs $V_{k} \in C, V_{k'}\in C'$ such that $(V_k,V_{k'})$ is $\varepsilon $-regular and $d(V_{k }, V_{k'}) > 1/2 + \theta $. As $|C| = N_1 \ge (2(\eta \theta )^{-1}h)^h$ this gives
	\begin{equation*}
		\sum _{\substack{V_k \in C, V_{k'}\in C':\\ (V_k,V_{k'}) \mbox{ } \varepsilon -reg}} |E(V_k, V_{k'})| 
		\leq \Big (\frac{1}{2} + \theta +\eta^{-1}hN_1^{-1/h}\Big )\Big | \bigcup _{V_k \in C} V_k \Big | \Big | \bigcup _{V_{k'} \in C'} V_{k'}\Big |
		\leq \Big (\frac {1}{2} + \frac{3 \theta }{2} \Big )\Big | \bigcup _{V_k \in C} V_k\Big | \Big | \bigcup _{V_{k'} \in C'} V_{k'}\Big |
	\end{equation*}
i.e. \eqref{equation: clique control} holds.

We can now complete the proof, upper bounding the number of forward edges of $T$. The $\varepsilon $-regular pairs $(V_{k}, V_{k'})$ with $V_{k} \in C \in {\cal C}_{\ell }$ and $V_{k'} \in C' \in \mathcal{C}_{\ell'}$ for some $\ell < \ell '$ contribute at most $\big (1/2 + 3\theta /2\big ) \binom {n}{2}$ forward edges to $T$ by \eqref{equation: clique control}. The remaining forward edges either (a) lie entirely in some set $U_{\ell }$, (b) contain a vertex from cluster $V_{k} \notin \bigcup _{\ell } \bigcup _{C \in {\cal C}_{\ell }} C$, (c) contain a vertex in $V_0$, or (d) lie between pairs $(V_{k},V_{k'})$ which are not $\varepsilon $-regular. The number of such edges is at most
	\begin{equation*}
		\sum _{\ell \in [L]}\binom {|U_{\ell }|}{2}  + 
		\sum _{\ell \in [L]} |{\cal R}_{\ell }\setminus (\cup _{C \in {\cal C}_{\ell }} C)| \bigg ( \frac{n^2}{K} \bigg )
		+ |V_0|n 
		+ \varepsilon \binom {K}{2} \bigg ( \frac {n}{K} \bigg )^2 
			\leq 
		\left (\frac{1}{2L} + 2\alpha + 3\varepsilon 
		+ \varepsilon \right ) \binom {n}{2} \leq \frac {\theta }{2} 
		\binom {n}{2}.
	\end{equation*}
Here we have used that $|{\cal R}_{\ell }\setminus (\cup _{C \in {\cal C}_{\ell }} C)| \leq \alpha W_{\ell } \leq 2\alpha K/L$ by part (ii) of the claim, that $L \geq 16/\theta $, $64\alpha = \theta $ and $16 \varepsilon \leq \theta $. Combined, these estimates show that $T$ has at most $\frac{1}{2}\binom {n}{2} + \theta n^2$ forward edges, as required.
\end{proof}

\section{Concluding remarks and open problems}
	We have shown that a large tournament $H$ is forcing 
	if and only if $H$ is transitive. Our main focus was on the stronger 
	property of being locally forcing. We proved that while many 
	tournaments do not satisfy this property 
	(Lemma \ref{lemma: properties of counting poly}) it
	does hold for many tournaments which we draw from a certain random distribution on nearly regular tournaments. The most natural model of random regular tournaments is to take a uniform distribution over all regular tournaments. We believe that in fact this also gives w.h.p.\ a \simple{} and hence (by Lemma \ref{lem: nearly regular is esimple} and Theorem \ref{thm:-simple and esimple iff locally forcing}) a locally forcing tournament. 
	
	Another result in this paper shows that a tournament is locally forcing if and only if it is \simple{} and \esimple. This in some sense says that in order to check whether a tournament $H$ is locally forcing one only needs to check whether models $\mathcal{T}(n,\alpha )$ or $\mathcal{T}(n,n,\alpha )$ can have the same local counts of $H$ as $\mathcal{T}(n,1/2)$ for some $\alpha \neq 1/2$.

	We actually believe that the \esimple{} assumption may be dropped entirely in Theorem \ref{thm:-simple and esimple iff locally forcing} either because it is implied by \simple{} or because it is satisfied by every tournament. In other words it would be interesting to determine if every \simple{} tournament is locally forcing. We reduced this question to the following problem about degree counting polynomials. Does there exist an $h$-vertex tournament $H$ such that the rescaled and recentered $a$-th order degree counting polynomials defined as
	$$q_{H,a}(x):=\binom{h}{a}^{-1}\sum_{A \in \binom{V(H)}{a}} (1+x)^{e(A,V \setminus A)}(1-x)^{e(V \setminus A,A)}-1$$
	have a common root in $(0,1]$, for all $1 \le a \le h-1$. We note that in a certain sense this is the correct tournament analogue of the Simonovits-S\'os conjecture from the induced graph case, discussed in the introduction.
	
	Our arguments rely on the regularity lemma. It is possible that one can find a nice class of tournaments for which one can show the locally forcing property directly, avoiding the use of the regularity lemma. For example, the tournament $Tr_3$ is not globally forcing but is locally forcing and this is not hard to see directly. Indeed, any tournament $T$ has at most $\frac12 \sum_{v \in T}\left( \binom{d^+(v)}{2}+\binom{d^-(v)}{2}\right)=\frac14 \sum_{v \in T} \frac{1}{2}\left((d^+(v)+d^-(v))^2+(d^+(v)-d^-(v))^2\right)-(n-1)=\frac18 n^3 (1+o(1))+\frac18 \sum_{v \in T} (d^+(v)-d^-(v))^2$ copies of $Tr_3$. So in order for the local counts to match that of the random tournament, by Cauchy-Schwarz, we must have $n^{-1}\big (\sum_{v \in U} |d^+_{T[U]}(v)-d^-_{T[U]}(v)|\big )^2 \leq \sum_{v \in U} (d^+_{T[U]}(v)-d^-_{T[U]}(v))^2=o(n^3)$ for any $U \subseteq V(T)$. This gives condition $\mathcal{P}_5$ from Chung and Graham \cite{CG} and implies $T$ must be quasirandom. Thus $Tr_3$ is locally forcing. Since any three vertices of a tournament either induce $Tr_3$ or $C_3$, if any subset of $V(T)$ has the correct count of $Tr_3$ then it has the correct count of $C_3$ as well. Thus $C_3$ is also locally forcing.
	
	We note that deciding whether a fixed tournament is \simple{} is a matter of counting how many zeros of the counting polynomial (minus a constant) one can find in the interval $[1/2,1]$. Since counting polynomial is of degree at most $\binom{h}{2}$ this can be done by Sturm's algorithm in polynomial time. Deciding whether a fixed tournament is \esimple{} can be done in a similar fashion, we first find the greatest common divisor of our degree counting polynomials and then find the number of roots of this greatest common divisor in $[1/2,1].$ While all of this can be done in polynomial time it is not clear how to compute the coefficients of our polynomials in polynomial time since they are defined in terms of vertex orderings of which there are $h!$ or subsets, in which case there are potentially as many as $\binom{h}{h/2}$. It could be interesting to determine if these terms could be calculated in polynomial time, as this would give a polynomial time algorithm for deciding whether a given tournament is locally forcing or not.
	
	Finally, it was brought to our attention that Fox, Himwich and Mani \cite{FHM} independently studied related problems in general directed graphs including certain forcing problems in this setting.

\section*{Acknowledgements}

We would like to thank Jan Volec for checking many small cases using a computer and Igor Balla for drawing our attention to \cite{AGK}.

\end{document}